\documentclass[12,reqno]{amsart}
\usepackage{amsmath}
\usepackage{amscd,amsthm,amsfonts,amsopn,amssymb,verbatim}
\usepackage{mathrsfs}
\numberwithin{equation}{section}

\newtheorem{theorem}{Theorem}[section]

\newtheorem{proposition}[theorem]{Proposition}

\newtheorem{lemma}[theorem]{Lemma}

\theoremstyle{remark}

\DeclareMathOperator{\supp}{supp\,}

\def\be{\begin{equation}}
\def\ee{\end{equation}}
\def\vp{\varphi}

\newcommand{\ltriple}{\lvert\lvert\lvert}
\newcommand{\rtriple}{\rvert\rvert\rvert}

\newcommand{\Bigltriple}{\Big\lvert\Big\lvert\Big\lvert}
\newcommand{\Bigrtriple}{\Big\rvert\Big\rvert\Big\rvert}

\allowdisplaybreaks
\begin{document}
\setlength{\parskip}{2pt}

\title[Almost sure global well posedness for radial NLS on the
3d ball]{Almost sure global well posedness for the radial Nonlinear
Schr\"odinger equation on the unit ball II: the 3d case}
\date{\today}
\author{Jean Bourgain}
\address{(J. Bourgain) Institute for Advanced Study, Princeton, NJ 08540}
\email{bourgain@math.ias.edu}
\author{Aynur Bulut}
\address{(A.Bulut) Institute for Advanced Study, Princeton, NJ 08540}
\email{abulut@math.ias.edu}
\thanks{The research of J.B. was partially supported by NSF grants DMS-0808042 and DMS-0835373 and the research of A.B. was supported by NSF under agreement No. DMS-0808042 and the Fernholz Foundation.}
\begin{abstract}
We extend the convergence method introduced in our works \cite{BB}--\cite{BB2}
for almost sure global well-posedness of Gibbs measure evolutions of the nonlinear
Schr\"odinger (NLS) and nonlinear wave (NLW) equations on the unit ball in 
$\mathbb{R}^d$ to the case of the three dimensional NLS.  This 
is the first probabilistic global well-posedness result for NLS with 
supercritical data on the unit ball in $\mathbb{R}^3$.

The initial data is taken as a Gaussian random process lying in the support 
of the Gibbs measure associated to the equation, and results are obtained almost 
surely with respect to this probability measure.  The key tools used include a 
class of probabilistic {\it a priori} bounds for finite-dimensional projections 
of the equation and a delicate trilinear estimate on the nonlinearity, which -- when 
combined with the invariance of the Gibbs measure -- enables the {\it a
priori} bounds to be enhanced to obtain convergence of the sequence of
approximate solutions.
\end{abstract}
\maketitle

\section{Introduction}

In the work at hand, we continue our study of Gibbs measure evolution for the
nonlinear Schr\"odinger (NLS) and nonlinear wave (NLW) equations on the unit
ball in Euclidean space, initiated in our earlier works \cite{BB}--\cite{BB2}.  In
particular, the aim of the present article is to extend the almost sure
global well-posedness result of \cite{BB2}, which was set on the unit ball in
$\mathbb{R}^2$, to the setting of the unit ball in $\mathbb{R}^3$.  The
techniques involved are a further development of the method introduced in our
work \cite{BB1} for the nonlinear wave equation, combined with a delicate
choice of function spaces adapted to the decay properties of the fundamental
solution of the Schr\"odinger equation.

More precisely, we shall consider the initial value problem for the cubic NLS
on the unit ball $B$ in $\mathbb{R}^3$,
\begin{align*}
\textrm{(NLS)}\quad\left\lbrace\begin{array}{rl}iu_t+\Delta u-|u|^2u&=0\\
u|_{t=0}&=\phi,
\end{array}\right.\label{(1.2)}
\end{align*}
where $u:I\times B\rightarrow\mathbb{C}$, subject to the Dirichlet boundary
condition $u|_{I\times \partial B}=0$, with randomly chosen radial initial
data $\phi$; the sense in which the randomization is taken will be specified
momentarily.

A fundamental property of (NLS) is that the equation takes the form of an
infinite dimensional Hamiltonian system,
\begin{align*}
iu_t=\frac{\partial H}{\partial \overline{u}},
\end{align*}
with conserved Hamiltonian
\begin{align*}
H(\phi)=\frac{1}{2}\int_B |\nabla \phi|^2+\frac{1}{4}\int_B |\phi|^{4}.
\end{align*}

In the case that the spatial domain $B\subset \mathbb{R}^3$ is replaced with
the $d$-dimensional torus $\mathbb{T}^d=\mathbb{R}^d/\mathbb{Z}^d$, a robust
theory of almost sure global well posedness for the Cauchy problem was
established in the seminal works \cite{B1}-\cite{B4} for a variety of general
classes of nonlinearities including both the attractive and repulsive
regimes; see also \cite{B5} for a brief survey of these results.  The
approach pioneered in this line of study was to obtain global control by
exploiting the invariant properties of the Gibbs measure inherent in the
Hamiltonian structure of the equation.

In preparation for our discussion below, we now outline the main steps of the
approach pursued in those works:
\begin{itemize}
\item[(i)] The first step is to consider a finite-dimensional projection of
the Cauchy problem for (NLS), allowing access to an invariant Gibbs measure
which gives global in time estimates for solutions.

\item[(ii)] A strong form of the local well-posedness theory driven by a
contraction mapping principle then allows to show convergence of solutions
for the finite-dimensional problems to a solution of the original equation.
The key point in this step is to obtain estimates which are uniform in the
projection parameter.

\item[(iii)] The two steps above are then combined to establish almost sure
global well-posedness for the original Cauchy problem, (NLS) with no
finite-dimensional projection.

\item[(iv)] The final step in the analysis is to establish the invariance of
the limiting Gibbs measure with respect to the evolution given by the
original, non-projected, (NLS) equation.
\end{itemize}

We remark that the local theory in this approach is a consequence of fixed
point arguments in suitable classes of function spaces.  Although such
results are usually available only for problems in which the initial data is
subcritical or critical with respect to the scaling of the equation,
nevertheless, the randomization gives additional integrability almost surely
in the random variable, and can often enable the application to classes of
supercritical data (see for instance \cite{B3} as well as \cite{BT12}).

On the other hand, in the setting of the present paper a more substantial
obstruction to implementing the approach described above is posed by the lack
of robust Strichartz estimates on domains with boundary, which renders the
fixed point technique ineffective for our purposes.  Indeed, in the current
work our arguments pursue a different path based on the treatment we
introduced for the three-dimensional nonlinear wave equation \cite{BB1} and adapted to
the two-dimensional NLS equation \cite{BB2}.  This approach is again
based on a procedure of finite dimensional projection, with the goal of
showing global well-posedness by establishing convergence for the sequence of
solutions of the projected equations.  However, with the fixed-point argument
unavailable, the proof of convergence follows from a more delicate analysis
of the fine behavior of solutions and their frequency interactions.

More precisely, the strategy in the present paper proceeds in the following
steps:
\begin{itemize}
\item[(i$'$)] Construction of a suitable collection of function spaces used
to establish convergence for the sequence of solutions to the
finite-dimensional projections.  Closely related to this is the
identification of the relevant embeddings and basic interpolation properties
of the spaces.

\item[(ii$'$)] Establishing {\it a priori} bounds for solutions of the
projected equations which remain uniform in the projection parameter.
\item[(iii$'$)] The formulation of an estimate of the contribution of the
nonlinearity.  This estimate is the most delicate stage in the process, and
serves to provide the decay necessary to establish convergence.

\item[(iv$'$)] The above ingredients are then combined to establish
convergence for the sequence of solutions of the projected
equations, almost surely in the randomization.  The limiting function is a
solution of the original equation and is defined for arbitrarily long time
intervals.
\end{itemize}

It is important to note that in our current setting the invariance of the
Gibbs measure is an essential ingredient in obtaining the short-time local
existence result, whereas in the fixed-point based approach of
\cite{B1}-\cite{B5} the local theory is developed independently of the
invariance of the Gibbs measure.  This is a major distinction between the two
approaches, and our use of the Gibbs measure at this stage of the argument
can be seen as the key piece of probabilistic information which allows to
overcome the lack of Strichartz estimates; for a complete discussion of this
issue we refer the reader to our treatment in \cite{BB1}, where the technique
was introduced.

Before giving the precise statement of our main results, we shall now
describe the finite-dimensional projections which form the basis of our
approach.

\subsection{Finite dimensional model and the Gibbs measure}

We shall consider solutions to the {\it truncated equation}
\be\label{(1.1)}
\left\lbrace\begin{array}{rl}iu_t+\Delta u-P_N(|u|^2u)&=0\\
u|_{t=0}&=P_N\phi,
\end{array}\right. 
\ee
where the operator $P_N$ is the projection to low frequencies defined by
\begin{align*} 
P_N\bigg(\sum_{n\in\mathbb{N}} a_ne_n(x)\bigg)=\sum_{n\leq N} a_ne_n(x).
\end{align*}
with $(a_n)\in \ell^2$ and $(e_n)$ as the sequence of radial eigenfunctions
of $-\Delta$ on $B$ with vanishing Dirichlet boundary conditions.

The initial value problem \eqref{(1.1)} is globally well-posed for
every integer $N\geq 1$: indeed, for any initial data $\phi\in L_x^2(B)$,
there exists a unique global solution $u_N:\mathbb{R}\times
B\rightarrow\mathbb{C}$ satisfying the associated {\it Duhamel formula},
\be
u_N(t)=e^{it\Delta}P_N\phi+i\int_0^t
e^{i(t-\tau)\Delta}P_N(|u_N|^2u_N)(\tau)d\tau.\label{(1.2)}
\ee

The {\it Gibbs measure} $\mu_G^{(N)}$ associated to \eqref{(1.1)} is defined
(up to normalization factors) by
\begin{align*}
\mu_G^{(N)}(A)&=\int_A \exp(-H_N(\phi))\prod_{i=1}^N d^2\phi\\
&=\int_A \exp\bigg(-\frac{1}{4}\lVert
P_N\phi\rVert_{L_x^4}^4\bigg)d\mu_F^{(N)}(\phi),\quad A\in\mathcal{M}
\end{align*}
where
\begin{align*}
H_N(\phi)=\frac{1}{2}\sum_{n\leq N} n^2|\widehat{\phi}(n)|^2+\frac{1}{4}\int_B
|P_N\phi(x)|^{4}dx.
\end{align*}
and $\mu_F^{(N)}$ is the free (Weiner) measure induced by the mapping
\begin{align*}
\Omega\ni \omega\mapsto \phi_\omega:=\sum_{n\leq N} \frac{g_n(\omega)}{n\pi}e_n,
\end{align*}
where $(g_n)$ is a sequence of IID normalized complex Gaussian random
variables.

As we will see below, basic facts concerning the sequence of eigenfunctions
$(e_n)$ ensure that the norms
\begin{align*}
\lVert \phi\rVert_{H_x^s(B)},\,\,\, s<\frac{1}{2}\quad \textrm{and}\quad \lVert
P_N\phi\rVert_{L_x^p(B)},\,\,\, p<6
\end{align*}
are finite $\mu_F^{(N)}$-almost surely for every $N\geq 1$.  These facts dictate
the spaces in which we look for solutions, and also serve to ensure that the
measure $\mu_G^{(N)}$ is well-defined, nontrivial and normalizable.  Finally, we
remark that $\mu_G^{(N)}$ is invariant under the evolution of the truncated
equation \eqref{(1.1)}, that is to say
\begin{align*}
\mu_G^{(N)}(\{\phi_\omega:\omega\in
\Omega\})=\mu_G^{(N)}(\{u_N(t):u_N\,\,\,\textrm{solves}\,\,\,\eqref{(1.1)}\,\,\,\textrm{with}\,\,\,\phi=\phi_\omega,\,\omega\in
\Omega\})
\end{align*}
for any $t\in\mathbb{R}$.

We are now ready to state the main result of this paper,
which establishes almost
sure convergence of the sequence of solutions to the truncated equation
\eqref{(1.1)} as the truncation parameter $N$ tends to infinity.

\vspace{0.1in}

\noindent
{\bf Theorem.}
{\it Let $(\Omega, p,\mathcal{M})$ be a given probability space.  For each $N\in
\mathbb{N}$, $\omega\in \Omega$ let $u_N$ denote the solution to
\eqref{(1.1)} with initial data $P_N\phi=P_N\phi^{(\omega)}$.  Then,
almost surely in $\omega$, for  
every $s<1/2$ and $T<\infty$, there exists $u_*\in C_t([0,T);H_x^s(B))$ such that $u_{N}$ converges to $u_*$ 
with respect to the norm $C_t([0,T);H_x^s(B))$.}

\vspace{0.1in}

The proof of the theorem follows the approach described above,  
and can be roughly outlined as consisting of the following steps: (1)
identification of the Fourier restriction spaces $X^{s,b}$ together with
a variant $X_{\ltriple\cdot\rtriple}$ as suitable classes
of function spaces, (2) the derivation of a family of {\it a priori} bounds
which are uniform in the finite-dimensional projection $P_N$, (3) a trilinear
estimate on the nonlinearity which allows to enhance the {\it a priori} bounds
into the decay necessary to establish convergence, and (4) a convergence
argument for $N\to\infty$ which assembles the above ingredients.

The first step in the analysis is the choice of function spaces.  
As is by now familiar in
the study of nonlinear dispersive equations, the spaces $X^{s,b}$ of
\cite{B-GAFA1, B-GAFA2} are the natural spaces to carry out perturbation theory
from the Duhamel formula \eqref{(1.2)}.  An additional component in the
analysis in the present work is the need to consider short time
intervals.  To balance this requirement with the degenerating constant in the
$X^{s,b}$-localization bound
$$
\lVert \psi f\rVert_{X^{s,b}}\lesssim \frac{1}{\delta^{b-1/2}}\lVert
f\rVert_{X^{s,b}},\quad b>\frac{1}{2},
$$
with $\psi(t)=\eta(t/\delta)$, $\delta>0$, where $\eta:\mathbb{R}\rightarrow
[0,1]$ is a smooth function such that $\eta=1$ on $[-1,1]$ and $\supp\eta\subset
[-2,2]$ (see, for instance, [7, Lecture 2]), 
we introduce also the slightly different space $X_{\ltriple\cdot\rtriple}$ 
for which the degenerating constant does not appear.

With the scale of function spaces identified, we next devote our attention to
{\it a priori} bounds for solutions of the truncated equations \eqref{(1.1)},
uniform in the truncation parameter.  We first obtain such
bounds in $L_x^pL_t^q$ norms, and subsequently extend the arguments to $X^{s,b}$
norms.  To obtain the almost sure global well-posedness result of the theorem,
it suffices to establish these bounds up to the exclusion of sets
of small measure in the statistical ensemble.  In view of this, the key
observation is that by exploiting the invariance of the Gibbs measure, it is
enough to establish analogous bounds for functions of the form
\begin{align*}
\sum_{n\in\mathbb{N}} \frac{g_n(\omega)}{n\pi}e_n(x).
\end{align*}

This enables us to combine standard estimates for Gaussian processes
and estimates on the eigenfunctions $e_n$ to obtain the desired bounds.

The next step is to obtain a trilinear estimate on the nonlinear term in the
Duhamel formula.  The argument to establish this bound
proceeds by decomposing each of the three linear factors appearing in the
nonlinearity $F(u)=|u|^2u$ into discrete frequencies and estimating the
resulting frequency interactions.  These estimates are performed using
space-time norms, $X^{s, b}$-spaces and further probabilistic considerations
based on the Gibbs measure invariance.
In fact, we need to distinguish several frequency regions where different
arguments apply.
Introducing these regions requires certain care.

The final step in establishing the theorem is to assemble the above
ingredients to show that the sequence $(u_N)$ of solutions to the truncated
equations \eqref{(1.1)} is almost surely a Cauchy sequence in the space
$C_t([0,T);H_x^s(B))$.  The core step in this argument takes the form of an
estimate for the $X_{\ltriple\cdot\rtriple}$ norm of the
difference $u_{N_1}-u_{N_0}$ for any integers $N_1\geq N_0\geq 1$.  This bound
is of the form 
\be\label{(1.3)}
\rtriple u_{N_1} -u_{N_0}\rtriple \lesssim N_0^{-c} \text { for some $c>0$}
\ee
for all $\omega\in\Omega$ outside a singular set having small measure.
The measures of these exceptional sets need to be sufficiently small in order to
deduce an almost everywhere convergence result.
Of course, large deviation estimates for Gaussian processes are essential here.
The final stage of the argument consists in revisiting the probabilistic claims
in order to justify the required quantitative form.

\section
{Notation and preliminaries}

Throughout our arguments we will frequently make use of a dyadic decomposition
in frequency, writing
\begin{align*}
f(x)=\sum_n \hat{f}(n)e_n(x)=\sum_{N\geq 1}\sum_{n\sim N} \hat{f}(n)e_n(x),
\end{align*}
where for each $n\in\mathbb{Z}$, the condition $n\sim N$ is characterized by
$N\leq n\leq 2N$.

For every $n\in\mathbb{N}$, define
\begin{align}\label{(2.1)}
e_n(x)=\frac{\sin(n\pi |x|)}{|x|}
\end{align}
and recall that $e_n$ is the $n$th radial eigenfunction of $-\Delta$ on $B$,
with associated eigenvalue $n^2$.  With this notation, we have the following
estimates on the norms of the eigenfunctions:
\begin{align}\label{(2.2)}
\lVert e_n\rVert_{L_x^p}\lesssim 1,\quad 1\leq p<3 
\quad\quad\quad\textrm{and}\quad\quad\quad
\lVert e_n\rVert_{L_x^p}\lesssim n^{1-\frac{3}{p}},\quad p>3,
\end{align}
along with the endpoint-type bound $\lVert e_n\rVert_{L_x^3}\lesssim (\log
n)^{1/3}$.  Moreover, the sequence $(e_n)$ also enjoys the following correlation
bound:
\begin{align}\label{(2.3)}
|c(n,n_1,n_2,n_3)|\lesssim \min\{n,n_1,n_2,n_3\},
\end{align}
where we have set
\begin{align}\label{(2.4)}
c(n,n_1,n_2,n_3)=\int_{B} e_n(x)e_{n_1}(x)e_{n_2}(x)e_{n_3}(x)dx.
\end{align}

Another essential tool in our analysis is the following probabilistic estimate
for sums of Gaussian random variables:
\begin{align}\label{(2.5)}
\bigg\lVert \sum_{n}
\alpha_ng_n(\omega)\bigg\rVert_{L^q(d\omega)}&\lesssim \sqrt{q}\big(\sum_{n}
|\alpha_n|^2\big)^{1/2},
\end{align}
where $(\alpha_n)\in \ell^2$, $2\leq q<\infty$, and $(g_n)$ is a sequence of
IID normalized complex Gaussians.

We also have the following multilinear version of the estimate \eqref{(2.5)}:
\be\label{(2.6)}
\bigg\lVert \sum_{n}
\alpha_nh_n(\omega)\bigg\rVert_{L^q(d\omega)}\lesssim (\sqrt{q})^k\bigg\lVert
\sum_{n} \alpha_nh_n(\omega)\bigg\rVert_{L^2(d\omega)}
\ee
for every $k\geq 1$, $2\leq q<\infty$ and each
$h_n$ is a product of at most $k$ Gaussians taken from a sequence $(g_n)$ 
as above.

As a consequence, if $(g_n)$ is a sequence of normalized IID complex Gaussian random variables, the bound
\begin{align}\label{(2.7)}
 \bigg\lVert \sum_{n} \alpha_n\cdot (|g_n(\omega)|^2-1)\bigg\rVert_{L^q(d\omega)}\lesssim
q\big(\sum_{n} |\alpha_n|^2\big)^{1/2}.
\end{align}
holds for every $(\alpha_n)\in\ell^2$ and $1\leq q<\infty$.

In the form \eqref{(2.6)}, we note that the inequality remains valid in the vector-valued
case, with $(\alpha_n)$ as elements of an arbitrary normed space $X$. See [12].

\subsection{Description of the function spaces}

Fix a time interval $I=[0,T)$ with $T>0$ sufficiently small, and let the space $X^{s,b}(I)$ denote the
class of functions $f:I\times B\rightarrow\mathbb{C}$ representable as
\be\label{(2.8)}
f(x,t)=\sum_{n,m} f_{n,m}e_n(x)e(mt),\quad (x,t)\in B\times I
\ee
for which the norm
$$
\lVert f\rVert_{s,b}:=\bigg(\sum_{n,m} \langle n\rangle^{2s}\langle
n^2-m\rangle^{2b}|f_{n,m}|^2\bigg)^{1/2}
$$
is finite, where the infinum is taken over all representations \eqref{(2.8)}.  We also refer the reader
to the works \cite{B-GAFA1}-\cite{B-GAFA2}, where these spaces were first introduced.

Moreover, when $f:I\times B\rightarrow\mathbb{C}$ has a representation \eqref{(2.8)}, we shall define the
function $T_{s,b}f$ via
\be\label{(2.9)}
(T_{s,b}f)(x,t)=\sum_{n,m}\langle n\rangle^{s}\langle n^2-m\rangle^{b}f_{n,m}e_n(x)e(mt).
\ee

Our analysis requires to consider short time intervals $[0, T]$, where $T$ will depend on the truncation
parameters.
In order to establish contractive estimates for the nonlinear term, we need a variant of the 
$\Vert \cdot \Vert_{0, \frac 12}$-norm adapted to the time interval.
We denote this norm by $\ltriple \cdot \ltriple_{0, \frac 12; T}$, and its
unit ball is generated by functions of the form
\be\label {(2.10)}
\sum_{n, m} \frac {a_{n, m}}{(|n^2-m|+\frac 1T)^{\frac 12}} \, e_n(x) e(mt) +\sum_{\substack {n, m\\
|n^2-m|>\frac 1T}} \frac {a_n}{|n^2-m|}\, e_n(x) e(mt)
\ee
with
$$
\sum_{n, m}|a_{n, m}|^2 \leq 1\quad\textrm{and}\quad\sum_n |a_n|^2\leq 1.
$$

Obviously, $\Vert \cdot \Vert_{0, b}\lesssim \ltriple\,\cdot\,\rtriple$ 
for $b<\frac 12$.  One can similarly introduce norms  
$\ltriple\,\cdot\,\rtriple_{s, \frac 12; T}$ for 
$s>0$, but we will not need them for our purposes.

The next few lemmas put into evidence some basic properties of the norm $\ltriple \cdot \rtriple$.

\begin{lemma}\label{Lemma1}
Let $\ltriple f\ltriple\leq 1$. Then
\be\label{(2.11)}
\frac 1T \int_0^T \Vert f(t) \Vert^2_{L^2_x} dt< O(1).
\ee
\end{lemma}

\begin{proof}
We first write $f$ as in \eqref{(2.10)}.  Then
\begin{align*}
\Vert f(t)\Vert^2_{{L^2_x}} &=\sum_n\bigg|\sum_m \frac {f_{n, m}}{(|n^2-m|+\frac 1T)^{\frac{1}{2}}}
\, e(mt)\bigg|^2+\sum_n \bigg|\sum_{|n^2 -m|>\frac 1T} \ \frac {e(mt)}{n^2-m}\bigg|^2|f_n|^2\\
&= (I)+ (II).
\end{align*}

Taking $0\leq \vp \leq 2$ such that $\vp\geq 1$ on $[0, 1]$ and supp\,$\hat\vp \subset [-1,1]$, we have
$$
\begin{aligned}
\int_0^T (I) dt\leq \int (I) \vp\Big(\frac tT\Big)dt&\leq T\sum_n \, \sum_{\substack{m, m',\\|m-m'|\leq \frac 1T}} \frac{|f_{n,m}|\,|f_{n,m'}|}{|n^2-m|+\frac{1}{T}}\\
&\leq T^2\sum_{|k|\leq \frac 1T} \ \sum_{n, m} |f_{n, m}| \, |f_{n, m+k}|\\
&\lesssim T\Vert f\Vert^2_{L_{t,x}^2}\\
&\lesssim T
\end{aligned}
$$
and similarly
$$
\begin{aligned}
\int_0^T (II)dt&\lesssim T\sum_n \sum_{\substack {m,m',\,|m-m'|\lesssim \frac 1T\\ 
|n^2-m|>\frac 1T,\,|n^2-m'|>\frac 1T}}\ \frac {|f_n|^2}{|n^2-m| \, |n^2-m'|}\\
&\lesssim \sum_{\substack{n, m\\ |n^2-m|>\frac 1T}} \ \frac {|f_n|^2}{|n^2-m|^2} \\
&\lesssim T.
\end{aligned}
$$
The combination of these two bounds suffices to prove the claim.
\end{proof}

The next statement expresses an important duality property with respect to
the Duhamel formula \eqref{(1.2)}.
\begin{lemma}\label{Lemma2}
Assume $\displaystyle f(x, t) =\sum_{{n\in\mathbb Z_+, m\in\mathbb Z}} f_{n, m} e_n(x) e(mt)$.
Then
$$
\Bigltriple \int^t_0 e^{i(t-\tau)\Delta} f(\tau) d\tau\Bigrtriple
\lesssim \max_{\ltriple g\ltriple_{0, \frac 12; T\leq 1}} \Big|\sum_{n,m}
f_{n, m} g_{n, m}\Big|,\eqno{(2.14)}
$$
where $\displaystyle g(x, t) =\sum_{\substack {n\in\mathbb{Z_+}, m\in \mathbb{Z}}} g_{n,m}e_n(x)e(mt)$.
\end{lemma}

\begin{proof}
Write
$$ \int^t_0 e^{i(t-\tau)\Delta} f(\tau)d\tau =\sum_{n,m} f_{n, m}\, e_n \frac {e(mt) - e(n^2t)}{m-n^2}
$$
and decompose this as
$$
\hspace{0.1in}\sum_{|m-n^2|>\frac 1T} \frac {f_{n, m}}{m-n^2} \, e_n \, e(mt)\eqno{(2.15)}\\
$$
$$
-\sum_{|m-n^2|> \frac 1T} \, \frac {f_{n, m}}{m-n^2} \, e_n \, e(n^2t)\eqno{(2.16)}\\
$$
$$
\hspace{0.3in} +\sum_{|m-n^2|\leq\frac 1T} f_{n, m} e_n\, \frac {e(mt)-e(n^2t)}{m-n^2}.\eqno{(2.17)}
$$
Hence we may write (2.15) as
$$
\sum_{|m-n^2|> \frac 1T} \frac{ b_{n, m}}{|m-n^2|^{\frac 12}} e_n \, e (mt)
$$
with
$$
b_{n, m} =\frac {\pm f_{n, m}}{|m-n^2|^{\frac 12}},
$$
which satisfies
\begin{align*}
\Big(\sum_{n,m}|b_{n, m}|^2\Big)^{\frac 12} &= \Big(\sum_{|m-n^2|>\frac 1T} \ \frac
{|f_{n, m}|^2}{|m-n^2|}\Big)^{\frac 12}
= \max\bigg|\sum_{|m-n^2|>{\frac 1T}} f_{n, m} \frac {a_{n, m}}{|m-n^2|^{\frac
12}}\bigg|
\end{align*}
where the maximum is over sequences $(a_{n,m}$) with
$$
\Big(\sum_{n,m} |a_{n, m}|^2\Big)^{1/2}\leq 1,
$$
which takes care of the contribution of (2.15) to the left-hand side of (2.14).

Next, let $$\vp(t) = \sum_k\hat\vp(k) e(kt) $$ satisfy $\vp =1$ on $[0, T],
\vp\geq 0$ together with the condition $|\hat\vp(k)|\lesssim \frac T{(1+|k|T)^2}$.

For $0\leq t\leq T$, write (2.16) as
$$
\Big[\sum_n b_n e_n e(n^2 t)\Big] \vp(t) =\sum_{n, k} b_n e_n
e\big((n^2+k)t\big)
\hat\vp (k)\eqno{(2.18)}
$$
with
$$
b_n=\sum_{|m-n^2|>\frac 1T} \ \frac {f_{n, m}}{m-n^2}.
$$
Thus (2.18) becomes
$$
\sum_{|m-n^2|\lesssim T} \ \frac {a_{n, m}}{|n^2-m|^{\frac 12}}
\, e_n \, e(mt)
$$
with $a_{n,m} =b_n|n^2-m|^{\frac 12}\hat\vp (n^2-m)$ and
\begin{align*}
\Big(\sum_{n, m} |a_{n, m}|^2\Big)^{\frac 12} &=\Big(\sum_n |b_n|^2\sum_k |k| \,
|\hat\vp(k)|^2\Big)^{\frac 12} \\
&\lesssim \Big(\sum_n |b_n|^2\Big)^{\frac 12}\\
&=\max \,\bigg|\sum_{n, m} \, f_{n, m}  \frac {a_n}{|m-n^2|} \chi_{|m-n^2|>\frac{1}{T}}\bigg|.
\end{align*}
with maximum taken over $(a_n)$ such that
\begin{align*}
\sum_n |a_n|^2\leq 1
\end{align*}
which is the desired estimate for the contribution of (2.16).

Finally, for $0\leq t\leq T$ and $\vp$ as above, write (2.17) as
$$
\sum_{|m-n^2|\leq\frac 1T} f_{n, m} \, e_n \, e(n^2t)
\frac {e\big((m-n^2)t\big)-1}{m-n^2} \vp (t)
$$
and expand the exponential in a power series
$$
\sum_{s\geq 1} \frac 1{s!} \Big[\sum_n b_n^{(s)} \, e_n \, e(n^2t)\Big]
\Big(\frac tT\Big)^s \vp(t)\eqno {(2.19)}
$$
with
$$
b_n^{(s)}=\sum_{|m-n^2|\leq \frac 1T} f_{n, m} (m-n^2)^{s-1}T^s
$$
to obtain
\begin{align*}
\Big(\sum_n|b_n^{(s)}|^2\Big)^{\frac 12}&\leq
T\bigg[\sum_n\Big(\sum_{|m-n^2|\leq \frac 1T}
|f_{n, m}|\Big)^2\bigg]^{\frac 12}
\leq \sqrt T\Big(\sum_{\substack{|m-n^2|\leq \frac 1T}} |f_{n, m}|^2\Big)^{\frac
12}.
\end{align*}

For each $s$, let $\psi_s(t) =\sum_k \hat\psi_s e(kt)$ be an extension of
$\big(\frac tT\big)^s$, $0\leq t\leq T$ such that $$|\psi_s|\leq 2\quad\textrm{and}\quad |\psi_s'|
\leq 10sT^{-1}.$$
Then $$|\widehat{\vp\psi_s}(k)| \leq\Vert\vp\psi_s\Vert_{L_x^1}\leq
2\Vert\vp\Vert_{L_x^1}\lesssim T$$ and $$\Vert\vp\psi_s\Vert_{H^{\frac 12}}\lesssim
\Vert\vp\psi_s\Vert_{L_x^2}^{\frac 12} (\Vert\vp'\psi_s\Vert_{L_x^2}
+\Vert\vp\psi_s'\Vert_{L_x^2})^{\frac 12} \lesssim T^{\frac 14} (T^{-\frac
12}+sT^{-\frac 12})^{\frac 12}\lesssim s^{\frac 12},$$
which in view of (2.19) gives the desired representation of (2.17).
\end{proof}

The norm $\ltriple \cdot \rtriple$ does not quite control the
$L^\infty_{0\leq t\leq T} L^2_x$-norm.
However, the following holds, which will suffice for our purpose.

\begin{lemma}\label{Lemma3}
Let $f$ and $g$ have expansions as in Lemma $\ref{Lemma2}$.  Then
$$
\Big|\sum_{n,m} f_{n,m} g_{n,m}\Big|\lesssim T\ltriple f\rtriple\cdot
\ltriple g\rtriple.
$$
\end{lemma}

\begin{proof}
From the representation \eqref{(2.10)}, we obtain
\begin{align*}
\sum_{n,m}|f_{n,m}| \ |g_{n,m}|&\lesssim \sum_{n, m} \frac{a_{n,m}b_{n,m}}{|n^2-m|+\frac 1T} +\sum_{|n^2-m|>\frac 1T}
\frac {a_n b_{n, m}}{|n^2-m|^{3/2}}\\
&\hspace{0.4in}+\sum_{|n^2-m|>\frac 1T}
\frac {a_{n, m} b_n}{|n^2-m|^{3/2}}+\sum _{\substack{|n^2-m|>\frac 1T}} \, \frac{a_nb_n}{|n^2-m|^2}
\end{align*}
with $$
\sum_{n,m}|a_{n,m}|^2\leq 1,\quad \sum_{n,m}|b_{n, m}|^2\leq 1,\quad \sum_n|a_n|^2\leq 1,
$$ and $$
\sum_n |b_n|^2\leq 1.$$

By the Cauchy-Schwarz inequality, the first term is bounded by $T$, while the 
second term is bounded by
$$
\Big\{\sum_n\Big(\sum_{\{m: |n^2-m|>\frac 1T\}} \ \frac
{|b_{n,m}|}{|n^2-m|^{3/2}}\Big)^2
\Big\}^{\frac 12} \lesssim T.
$$

The estimate for the third term is similar.  Estimating the last term,
we obtain the bound
$$T\sum_n a_nb_n\lesssim T,$$
which allows to complete the lemma.
\end{proof}

Next, we establish several inequalities bounding suitable $L_x^p L_t^q$-norms
in terms of $X^{s, b}$-norms.  These will be essential to our analysis.
\begin{lemma}\label{Lemma4}
The spaces $X^{s,b}$ obey the following embedding relations:
\begin{enumerate}
\item[(i)] For $2<p<3$ and $b_1>\frac{1}{4}$,
\begin{align*}
\lVert f\rVert_{L_x^{p}L_t^2}\lesssim \lVert f\rVert_{{0,b_1}}.
\end{align*}
\item[(ii)]  For $3<p<6$, $s>1-\frac{3}{p}$ and $b_2>\frac{1}{2}$,
\begin{align*}
\lVert f\rVert_{L_x^pL_t^4}\lesssim \lVert f\rVert_{{s,b_2}}.
\end{align*}
\item[(iii)] For $\frac{1}{4}<b_3<\frac{1}{2}$ and $\epsilon>0$,
\begin{align*}
\lVert f\rVert_{L_x^3L_t^\frac{4}{3-4b_3}}\lesssim \lVert
f\rVert_{{\epsilon,b_3}}.
\end{align*}
\item[(iv)] For $b_4>\frac 12$ and $s>\frac 12$,
\begin{align*}
\Vert f\Vert_{L_x^3 L_t^\infty}\lesssim \Vert f\Vert_{{s, b_4}}.
\end{align*}
\item[(v)] For $3\leq p\leq 6, 4\leq q\leq \infty, s>\frac 32-\frac 3p-\frac 2q$ 
and $b_5>\frac 12$
\begin{align*}
\Vert f\Vert_{L_x^pL_t^q}\lesssim \Vert f\Vert_{{s, b_5}}.
\end{align*}
\item[(vi)] For $\frac14<b_6<\frac 12, 3<p< \frac 6{3-4b_6}, \frac 4{3-4b_6}
< q<\infty$, and $s> \frac 52- \frac 3p-\frac 2q-2b_6$,
\begin{align*}
\Vert f\Vert_{L^p_x L_t^q} \lesssim \Vert f\Vert_{{s, b_6}}.
\end{align*}
\item [(vii)] For $2\leq p<\frac 83$ and $b_7>\frac 12$, 
\begin{align*}
\Vert f\Vert_{L^p_xL^p_t}\lesssim \Vert f\Vert_{{0, b_7}}.
\end{align*}
\item[(viii)] For $\frac 14 < b_8<\frac 12, p<\frac {24}{4b_8+7}$, and
$q<\frac 8{5-4b_8}$,
\begin{align*}
\Vert f\Vert_{L^p_xL^q_t}\lesssim \Vert f\Vert_{{0, b_8}}.
\end{align*}
\end{enumerate}
\end{lemma}

\begin{proof}
We begin with (i).  Let $2<p<3$ be given.  Then for every $f$ as in
\eqref{(2.8)}, applying the Plancherel identity in time followed by the
Minkowski inequality, the eigenfunction estimate \eqref{(2.4)} and the
Cauchy-Schwarz inequality, we have
\begin{align*}
\lVert f\rVert_{L_x^pL_t^2}&\lesssim \bigg(\sum_m \bigg\lVert \sum_n f_{m,n}
e_n(x)\bigg\rVert_{L_x^p}^2\bigg)^{1/2}\\
&\lesssim \bigg(\sum_m\bigg(\sum_n |f_{m,n}|\bigg)^2\bigg)^{1/2}\\
&\lesssim \bigg(\sum_m \bigg(\sum_n \langle
m-n^2\rangle^{2b}|f_{m,n}|^2\bigg)\bigg(\sum_n \frac{1}{\langle
m-n^2\rangle^{2b}}\bigg)\bigg)^{1/2}.
\end{align*}
Observing that $b>\frac{1}{4}$ implies
\begin{align*}
\sup_m\sum_n \frac{1}{\langle m-n^2\rangle^{2b}}<\infty
\end{align*}
then establishes (i) as desired.

We now turn to (ii), for which we argue as in the proof of \cite[Lemma $2.2$]{BB2}.  
Let $3<p<6$ be given.  Then, writing \eqref{(2.8)} in the form
\begin{align*}
f(t,x)&=\sum_{m}\bigg(\sum_n f_{m+n^2,n}e_n(x)e(n^2t)\bigg)e(mt),
\end{align*}
we perform a dyadic decomposition into intervals $m\sim M$, $n\sim N$, expand
the square inside the norm $\lVert \,|\cdot
|^2\,\rVert_{L_x^{p/2}L_t^2}^{1/2}$, and use the Plancherel identity in the
$t$ variable to obtain
$$
\begin{aligned}
\nonumber \lVert f\rVert_{L_x^pL_t^4}&\lesssim \sum_{M,N}\sum_{m\sim M}
\bigg\lVert \sum_{n\sim N}
f_{m+n^2,n}e_n(x)e(n^2t)\bigg\rVert_{L_x^pL_t^4}\\
\nonumber&\lesssim \sum_{M,N} \sum_{m\sim M} \bigg\lVert
\bigg(\sum_{\ell}\bigg|\sum_{\stackrel{n,n'\sim N}{n^2+(n')^2=\ell}}
f_{m+n^2,n}f_{m+(n')^2,n'}e_n(x)e_{n'}(x)\bigg|^2\bigg)^{1/2}\bigg\rVert_{L_x^{p/2}}^{1/2}\\
&\lesssim \sum_{M,N} \sum_{m\sim M}
\bigg(\sup_{\ell}\sum_{\stackrel{n,n'\sim
N}{n^2+(n')^2=\ell}}1\bigg)^{1/4}\bigg\lVert \bigg(\sum_{n\sim N}
|f_{m+n^2,n}|^2e_n(x)^2\bigg)\bigg\rVert_{L_x^{p/2}}^{1/2}.
\end{aligned}
\eqno{(2.20)}
$$
where we have used the Cauchy-Schwarz inequality to obtain the last bound.

Note that arithmetic considerations associated with lattice points on circles
(see, for instance [10, Lemma 2.1] and the comments in the proof of
[10, Lemma 2.2]) entail the bound
$$
\sup_{\ell\geq 0} \, \Big|\{(n,n')\in [0,N]^2:n^2+(n')^2=\ell\}\Big|\lesssim
N^\epsilon\eqno{(2.21)}
$$
for any $\epsilon>0$ (where the implicit constant may depend on $\epsilon$).

Set $\epsilon=s-(1-\frac{3}{p})$.  Then, using (2.21) followed by
the Minkowski inequality, the eigenfunction estimates (2.21),
and the Cauchy-Schwarz inequality in the summation over $m\sim M$,
\begin{align*}
(2.20) &\lesssim \sum_{M,N} \sum_{m\sim M} N^{\epsilon/4} \bigg(\sum_{n\sim N}
|f_{m+n^2,n}|^2\lVert e_n(x)\rVert_{L_x^{p}}^2\bigg)^{1/2}\\
&\lesssim \sum_{M,N} \sum_{m\sim M} N^{\epsilon/4} \bigg(\sum_{n\sim N}
n^{2-\frac{6}{p}}|f_{m+n^2,n}|^2\bigg)^{1/2}\\
&\lesssim \sum_{M,N} N^{\epsilon/4}M^{\frac{1}{2}}\bigg(\sum_{m\sim M}
\sum_{n\sim N} n^{2-\frac{6}{p}}|f_{m+n^2,n}|^2\bigg)^{1/2}\\
&\lesssim \sum_{M,N} N^{-3\epsilon/4}M^{\frac{1}{2}-b}\bigg(\sum_
{\stackrel{n\sim N}{m-n^2\sim M}} \langle n\rangle^{2s}\langle m-n^2
\rangle^{2b}|f_{m,n}|^2\bigg)^{1/2}\\
&\lesssim \Vert f\Vert_{s, b},
\end{align*}
since
$$
\sum_{M, N} N^{-3\epsilon /2} M^{1-2b}<\infty.
$$
This completes the proof of part (ii) of the lemma.

The inequality stated in part (iii) now follows from parts (i) and (ii) by
standard interpolation arguments.

Next, we prove (iv).
Since $b_4>\frac 12$, it suffices to consider $f$ of the form
$$
f(x, t) =\sum_n a_n e_n(x) e(n^2 t),
$$
with 
$$\sum_n n^{2s} |a_n|^2\leq 1.$$
It follows from the Cauchy-Schwarz inequality that for any $\epsilon>0$
$$
|f(x, t)|\lesssim \Big[\sum_n |a_n|^2 n^{1+\epsilon}|e_n(x)|^2\Big]^{\frac 12}
$$
and hence
$$
\Vert f\Vert_{L_x^3L^\infty_t} \lesssim \Big[\sum_n|a_n|^2 n^{1+\epsilon} \Vert
e_n\Vert_{L_x^3}^2\Big]^{\frac 12}< O(1).
$$

Inequality (v) then follows by interpolation between (ii) and (iv), while (vi) is
obtained by interpolating between (i) and (v).

We prove (vii), taking $f$ of the form
$$
f(x, t) =\sum_n a_ne_n(x) e(n^2 t)=\sum_n a_n \frac {\sin\pi nr}r \, e(n^2 t)
$$
with $r=|x|$ and $\sum_n|a_n|^2\leq 1$.

Fix $0<\rho\leq 1$ and consider values of $x$ in the annulus $\displaystyle\frac{\rho}{2}\leq r\leq\rho$.
We make two estimates.  We first note that 
\begin{align}
\nonumber\Vert f\Vert_{L^4_{|x|\sim \rho}L_t^4}&\leq \frac 1{\sqrt\rho} \Big\Vert\sum_n
a_n(\sin \pi nr) e(n^2t) \Big\Vert_{L^4_{r\leq 1} L^4_{|t|\leq 1}}\\
&\nonumber\leq\frac 1{\sqrt \rho} \bigg[ \max_{k,\ell} \Big|\Big\{(n,n')\in\mathbb Z^2;
n\pm n'=k, n^2+(n')^2 =\ell\Big\}\Big|\bigg]^{\frac 14}\\
&\lesssim \frac 1{\sqrt \rho}. 
\tag{2.22}
\end{align}

On the other hand, one has
$$
\Vert f\Vert_{L^2_{|x|\sim\rho} L_t^2} \leq \Big\Vert\sum_n a_n(\sin \pi nr)
e(n^2t)\Big\Vert_{L^2_{r\sim\rho}L^2_{|t|<1}}\lesssim \sqrt\rho.\eqno{(2.23)}
$$

Hence (vii) follows by interpolation between (2.22), (2.23) and summation over
dyadic $\rho=2^{-j}$.

Finally, (viii) is obtained by interpolation between (i) and (vii).  This completes the proof of Lemma 2.4.
\end{proof}

\section{A priori uniform bounds}

In this section, we establish $X^{s,b}$ bounds on solutions of the truncated
equation (1.1) which are uniform in the truncation parameter
$N$.  For this purpose, we will first obtain a preliminary uniform estimate
on the norms $L_x^pL_t^q$ for suitable values of $p$ and $q$.  In particular,
we have the following:
\begin{lemma}\label {Lemma5}
For every $0\leq s<1/2$, $1\leq p<\frac{6}{1+2s}$, $1\leq q<\infty$, 
there exists a constant $C>0$ such that for every $N>0$ one has the bound
\be
\mu_F^{(N)}(\{\phi:\lVert
(\sqrt{-\Delta})^su\rVert_{L_x^pL_t^q}>\lambda\})\lesssim
\exp(-c\lambda^c),
\label{(3.1)}
\ee
where $u=u_N$ is a solution to the truncated equation \eqref{(1.1)}
associated to initial data $\phi$ (truncated as $P_N\phi$).
\end{lemma}

\begin{proof}
Without loss of generality we may assume $p>3$.  It suffices to show 
that \eqref{(3.1)} holds with $\mu_F^{(N)}$ replaced by the Gibbs 
measure $\mu_G$.  Indeed, suppose that one has
\be\label{(3.2)}
\mu_G(A_\lambda)\leq C\exp(-c\lambda^c) 
\ee
with
\begin{align*}
\nonumber A_\lambda:=\{\phi:\lVert
(\sqrt{-\Delta})^su\rVert_{L_x^pL_t^q}>\lambda\},\quad \lambda>0.
\end{align*}
Then, fixing $\lambda_1>0$, we have
\begin{align}
\mu_F^{(N)}(A_\lambda)&=\mu_F^{(N)}(A_\lambda \cap
\{\phi:\lVert\phi\rVert_{L_x^4}>\lambda_1\})+\mu_F^{(N)}(A_\lambda \cap
\{\phi:\lVert\phi\rVert_{L_x^4}\leq\lambda_1\})\nonumber\\
&\lesssim \mu_F^{(N)}(\{\phi:\lVert\phi\rVert_{L_x^4}>\lambda_1\})+\exp\bigg(
\frac{1}{4}\lambda_1^4\bigg)\mu_G(A_\lambda)\nonumber \\
&\lesssim \mu_F^{(N)}(\{\phi:\lVert\phi\rVert_{L_x^4}>\lambda_1\})+
\exp\bigg(\frac{1}{4}\lambda_1^4\bigg)\exp(-c\lambda^c).\label{(3.3)}
\end{align} 
To estimate the first term in \eqref{(3.3)}, we fix $q_1\geq 4$ and
appeal to the Tchebyshev and Minkowski inequalities followed by the estimate
\eqref{(2.5)} on sums of Gaussian random variables.  This gives
\begin{align}
\nonumber \mu_F^{(N)}(\{\phi:\lVert \phi\rVert_{L_x^4}>\lambda_1\})&\lesssim
\frac{1}{\lambda_1^{q_1}}\Big[\mathbb{E}_{\mu_F^{(N)}}\lVert
\phi\rVert_{L_x^4}^{q_1}\Big]\\
\nonumber &\leq \frac{1}{\lambda_1^{q_1}}\Big\lVert
\bigg(\mathbb{E}_{\mu_F^{(N)}}\Big[\Big(\sum_{n}
\frac{g_n(\omega)}{n}e_n(x)\Big)^{q_1}\Big]\bigg)^{1/q_1}\Big\rVert_{L_x^4}^{q_1}\\
\nonumber &\lesssim \Big(\frac{\sqrt{q_1}}{\lambda_1}\Big)^{q_1}\Big\lVert
\Big(\sum_n \frac{|e_n(x)|^2}{n^2}\Big)^{1/2}\Big\rVert_{L_x^4}^{q_1}\\
&\lesssim
\Big(\frac{\sqrt{q_1}}{\lambda_1}\Big)^{q_1}\Big(\sum_n \frac{\lVert
e_n\rVert_{L_x^4}^2}{n^2}\Big)^{q_1/2}.\label{(3.4)}
\end{align}
where in obtaining the last inequality we have used the Minkowski inequality.

Invoking now the eigenfunction estimate \eqref{(2.2)},
\begin{align*}
\eqref{(3.4)}&\lesssim
\Big(\frac{\sqrt{q_1}}{\lambda_1}\Big)^{q_1}\Big(\sum_n
n^{-3/2}\Big)^{q_1/2}\lesssim \Big(\frac{\sqrt{q_1}}{\lambda_1}\Big)^{q_1}
\end{align*}

We therefore obtain
\begin{align*}
\mu_F^{(N)}(A_\lambda)&\lesssim
\Big(\frac{\sqrt{q_1}}{\lambda_1}\Big)^{q_1}+\exp(\frac{1}{4}\lambda_1^4)\mu_G(A),
\end{align*}
so that optimizing in the choice of $q_1$ gives
\begin{align*}
\mu_F^{(N)}(A_\lambda)&\lesssim \exp(-c\lambda_1^c)
\end{align*}
as desired.

It therefore suffices to show \eqref{(3.2)}, which we recall was the
desired inequality with the measure $\mu_F^{(N)}$ replaced by the (invariant)
Gibbs measure $\mu_G=\mu_G^{(N)}$.  We argue as above: fixing $q_2\geq
\max\{p,q\}$ and invoking the Tchebychev and Minkowski inequalities, one has
\begin{align}
\nonumber\mu_G(A_\lambda)&\leq \lambda^{-q_2}\mathbb{E}_{\mu_G}\big[\lVert
(\sqrt{-\Delta})^s u\rVert_{L_{x}^pL_t^q}^{q_2}\big]\\
&\lesssim \lambda^{-q_2}\Big\lVert \Big(\mathbb{E}_{\mu_G}
[((\sqrt{-\Delta})^s u)^{q_2}]\Big)^{1/q_2}\Big\rVert_{L_x^pL_t^q}^{q_2}
\label{(3.5)}
\end{align}

Now, using the invariance of the Gibbs measure $\mu_G=\mu_G^{(N)}$ with
respect to the truncated evolution (with $u=u_N$ being a solution of the
truncated equation) followed by the estimate for sums of Gaussian random
variables given by \eqref{(2.5)},
\begin{align*}
\eqref{(3.5)}&\lesssim\lambda^{-q_2}\bigg\lVert
\bigg(\mathbb{E}_{\mu_G} \Big[\Big(\sum_n
\frac{g_n(\omega)}{n^{1-s}}e_n\Big)^{q_2}\Big]\bigg)^{1/q_2}\bigg\rVert_{L_x^pL_t^q}^{q_2}\\
&\lesssim
\Big(\frac{\sqrt{q_2}}{\lambda}\Big)^{q_2}\bigg\lVert \sum_n
\frac{|e_n(x)|^2}{n^{2(1-s)}}\bigg\rVert_{L_x^{p/2}}^{q_2/2}.
\end{align*}
To conclude, we use the eigenfunction estimate \eqref{(2.2)} together
with the condition $p<\frac{6}{1+2s}$ to get the bound
\begin{align*}
\bigg\lVert \sum_n
\frac{|e_n(x)|^2}{n^{2(1-s)}}\bigg\rVert_{L_x^{p/2}}^{q_2/2}&\lesssim
\bigg(\sum_n \frac{1}{n^{2(1-s)}}\lVert
e_n(x)\rVert_{L_x^p}^2\bigg)^{q_2/2}\\
\nonumber &\lesssim \bigg(\sum_n n^{2(s-\frac{3}{p})}\bigg)^{q_2/2}\\
&\lesssim 1.
\end{align*}

Hence
\be
\label{(3.6)}\mu_G(A_\lambda)\lesssim
\Big(\frac{\sqrt{q_2}}{\lambda}\Big)^{q_2}
\ee
Optimizing the choice of $q_2$ in \eqref{(3.6)} as for $q_1$ above gives
\begin{align*}
\mu_F^{(N)}(\{\phi:\lVert
(\sqrt{-\Delta})^su\rVert_{L_x^pL_t^q}>\lambda\})&\lesssim \exp(-c\lambda^2)
\end{align*}
as desired.
\end{proof}

We are now ready to establish uniform $X^{s,b}$-bounds.

\begin{proposition}
\label{Proposition6}
Fix $0\leq s<\frac{1}{2}$ and $\frac{1}{2}<b<\frac{3}{4}$.  Then there exists
$C>0$ such that for all $N>0$, if $u=u_N$ is a solution to the truncated
equation \eqref{(1.1)}, then 
\begin{align*}
\mu_F^{(N)}\bigg(\bigg\{\phi:\lVert u\rVert_{s,b}>
\lambda\bigg\}\bigg)\lesssim \exp(-c_1\lambda^{c_2}).
\end{align*}
\end{proposition}

\begin{proof}
Let $s\in [0,\frac{1}{2})$ and $b\in (\frac{1}{2},\frac{3}{4})$ be given.
Fix $N\geq 1$ and write the Duhamel formula
\be
u(t)=e^{it\Delta}\phi+\int_0^t
e^{i(t-\tau)\Delta}|u|^2u(\tau)d\tau.\label{(3.7)}
\ee

We estimate both the linear and nonlinear terms in \eqref{(3.7)}
individually.  We begin with the linear term. Let $T_{s,b}$ be the operator
defined in \eqref{(2.9)}.  Then, fixing $q\geq 2$ and invoking the
Tchebychev and Minkowski inequalities, one has
\begin{align*}
\nonumber \mu_F^{(N)}(\{\phi:\lVert
e^{it\Delta}\phi\rVert_{s,b}>\lambda\})&\leq
\lambda^{-q}\mathbb{E}_\omega\Big[\lVert
T_{s,b}e^{it\Delta}\phi\rVert_{L_{t,x}^2}^q\Big]\\
\nonumber &\lesssim \lambda^{-q}\lVert \mathbb{E}_\omega
[(T_{s,b}e^{it\Delta}\phi)^{q}]^{1/q}\rVert_{L_{t,x}^2}^q\\
\nonumber &\lesssim \lambda^{-q}q^{q/2}\bigg\lVert \sum_n
\frac{|e_n(x)|^2}{n^{2(1-s)}}\bigg\rVert_{L_{t,x}^1}^{q/2}\\
&\lesssim \lambda^{-q}q^{q/2}.
\end{align*}
Appropriate choice of $\lambda$ gives
\be\label{(3.8)}
\mu_F^{(N)}(\{\phi:\lVert e^{it\Delta}\phi\rVert_{s,b}>\lambda\})\lesssim
\exp(-c\lambda^2).
\ee

Turning to the integral term, we set $f=|u(\tau)|^2u(\tau)$ and observe that
the expansion
$\displaystyle f(x,\tau)=\sum_{m,n} f_{n,m}e_n(x)e(m\tau)$ leads to
\begin{align*}
\int_0^t e^{i(t-\tau)\Delta}f(\tau)d\tau&=\int_0^t
\bigg(\sum_{m,n}f_{n,m}e_n(x)e((t-\tau)n^2+m\tau)\bigg)d\tau\\
&=\sum_{m,n} \frac{if_{n,m}}{(n^2-m)}e_n(x)(e(tn^2)-e(tm)).
\end{align*}
Applying H\"older's inequality and recalling $b>\frac{1}{2}$, we obtain
\begin{align*}
\bigg\lVert \int_0^t
e^{i(t-\tau)\Delta}f(\tau)d\tau\bigg\rVert_{s,b}&\lesssim
\bigg(\sum_{n,m} \frac{\langle n\rangle^{2s}|f_{n,m}|^2}{\langle
n^2-m\rangle^{2(1-b)}}\bigg)^{1/2}\\
&=\sup_{\stackrel{v\in X^{0,1-b}}{\lVert v\rVert_{0,1-b}\leq 1}}\,\,
\bigg|\int_0^1\int_{B} v(t,x)(\sqrt{-\Delta})^{s}f(t,x)dxdt\bigg|\\
&\lesssim \sup_{\stackrel{v\in X^{0,1-b}}{\lVert v\rVert_{0,1-b}\leq 1}}
\,\,\lVert v\rVert_{L_x^{3-\epsilon}L_t^2}\lVert
(\sqrt{-\Delta})^{s}u\rVert_{L_x^{\frac{3-\epsilon}{1-\epsilon}}L_t^6} \lVert
u\rVert_{L_x^{6-2\epsilon}L_t^6}^2
\end{align*}
Now, invoking Lemma $\ref{Lemma4}$ (i) in the form
$$
\lVert v\rVert_{L_x^{3-\epsilon}L_t^2}\lesssim \lVert v\rVert_{0,1-b},
$$
and using Lemma $\ref{Lemma5}$ to estimate the norms of $u$,
\begin{align*}
\bigg\lVert \int_0^t
e^{i(t-\tau)\Delta}f(\tau)d\tau\bigg\rVert_{s,b}&\lesssim \lambda^3
\end{align*}
for each $\lambda>0$ and all $\omega\in \Omega$ outside a set of measure
$O(\exp(-c\lambda^c))$.

We therefore have (adjusting the value of the constant $c$ as well as the
implicit constant)
\be\label{(3.9)}
\mu_F^{(N)}\bigg(\bigg\{\phi:\bigg\lVert \int_0^t
e^{i(t-\tau)\Delta}fd\tau\bigg\rVert_{s,b}>\lambda\bigg\}\bigg)\lesssim
\exp(-c\lambda^c).
\ee

To conclude, collecting \eqref{(3.8)} and \eqref{(3.9)},
\be\label{(3.10)}
\mu_F^{(N)}\bigg(\{\phi:\lVert u\rVert_{s,b}> \lambda\}\bigg)
\lesssim \exp(-c\lambda^c)
\ee
which gives the desired inequality.
\end{proof}

\section{The nonlinear term}

The main issue is an estimate on the $\ltriple\,\cdot\,\rtriple$-norm of 
trilinear expressions of the form
\be\label{(4.1)}
\int_0^t e^{i(t-\tau)\Delta} P_N[P_{N_1} U^1\, \overline{P_{N_2}u^{(2)}} \, P_{N_3}
u^{(3)}](\tau) d\tau
\ee
with $t<T$, where $U^1$ belongs to $X_{\ltriple\,\cdot\,\rtriple}$ and 
$u^{(2)}, u^{(3)}: \mathbb R\times B\to \mathbb C$ are solutions to truncated equations 
\eqref{(1.1)} for possibly different truncations $N^{(2)} \geq N_2, N^{(3)}\geq N_3$
and initial data
$$
u^{(i)} \big|_{t=0} =P_{N^{(i)}} (\phi),\quad i=2, 3.
$$

In order to establish a contractive estimate on {(4.1)}, $T$ will have to be chosen
sufficiently small; more specifically, we shall require
\be\label{(4.2)}
T\sim \frac 1{\log N_*} \text { with } N_*=\max (N_1, N_2, N_3).
\ee

As will be clear later on, this choice of $T$ is essential in our argument due to the presence of
a certain logarithmic divergence.

Our analysis is based on $L_x^pL_t^q$ norms as well as the norms $\Vert \cdot \Vert_{s, b}$ and $\ltriple\,\cdot\,\rtriple$.  Various 
contributions are considered, requiring different arguments.  While the norms
 $\Vert \cdot \Vert_{s, b}$ and $\ltriple\,\cdot\,\rtriple$ allow in particular for Fourier
 restrictions of the form $\chi_{[ (n^2 -m)\lesssim K]}$, these operations are in general not
allowed for $ L_x^pL_t^q$ norms.  For this reason, certain care is required in organizing the 
argument.

We denote by $N, N_i$, $i=1,2,3$ integers of the form $2^j$ and $n\sim N_i$ means $N_i \leq n< 2N_i$.
Denote $u_2=P_{N_2} u^{(2)}$ and $u_3=P_{N_3}u^{(3)}$.

We start by applying Lemma \ref{Lemma2} and estimate $\ltriple \eqref{(4.1)}\rtriple$ by
\be
\int_0^1\int_B \bar v (P_{N_1} U^1)\bar u_2 u_3 dxdt
\label{(4.3)}
\ee
with
$$
\ltriple v\rtriple \leq 1.
$$
By Cauchy-Schwarz,
\be\label{(4.4)}
\eqref{(4.3)} \leq \Big[\iint |v|^2 |u_2|^2 dxdt\Big]^{\frac 12} \Big[\iint |P_{N_1} U^1|^2 |u_3|^2
dxdt\Big]^{\frac 12}.
\ee

In each factor on the right-hand side of \eqref{(4.4)}, $u^{(2)}$ and $u^{(3)}$ are obtained from the 
same truncated equation.  This is essential for our analysis.

We have therefore reduced the estimate of \eqref{(4.3)} to estimating
\be
\iint \overline{P_{N}v} P_{N_1} v_1 \, \overline{P_{N_2}u} \, P_{N_3}u\label{(4.5)}
\ee
with $u$ obtained from some truncated equation \eqref{(1.1)} and
$$
\ltriple v\rtriple\leq 1, \quad\quad \ltriple v_1\rtriple\leq 1.
$$

Write \eqref{(4.5)} as 
\be
\sum_{\substack {n\leq N, n_i\leq N_i\\ m-m_1+m_2-m_3=0}} 
\overline{\hat v (n, m)} \ v_1(n_1, m_1) \, \overline{\hat u(n_2, m_2)}\, \hat u(n_3, m_3) \,
c(n, \bar n)\label{(4.6)}
\ee
with
$$
c(n,\bar n)=c(n, n_1 , n_2, n_3).
$$
Subdividing $[0, N_i]$ into dyadic intervals $[N_i', 2N_i']$, we estimate
\be
\eqref{(4.6)}\leq \sum_{N_2', N_3'}\Bigg| \sum_{\substack
{n\leq N, n_1 \leq N_1,  n_2\sim N_2', n_3\sim
N_3'\\ m-m_1+m_2-m_3=0}} c(n,\bar n)A_{n,m,\bar n,\bar m}\Bigg|\label{(4.7)}
\ee
with
\begin{equation*}
A_{n,m,\bar n,\bar m}=\overline{\hat v (n, m)} \ \hat v_1(n_1, m_1) \, \overline{\hat u(n_2, m_2)}\, \hat u(n_3, m_3).
\end{equation*}

Fix $N_2', N_3'$ and assume $N_2'\geq N_3'$. Set
$$
K=(N_2')^{10^{-3}}
$$
and define
\be
c_K (n, \bar n)= \begin{cases} c(n, \bar n),&\text { if }\quad |n^2-n_1^2 +
n_2^2 - n_3^2|< 10K\\
0,&\text { otherwise.}\end{cases}\label{(4.8)}
\ee

We now estimate
\begin{align}
&\Bigg|\sum_{\substack{ n\leq N, n_1\leq N_1, n_2\sim N_2', n_3\sim N_3'\\
m-m_1+m_2 -m_3 =0}}c(n,\bar n)A_{n,m,\bar n,\bar m} \Bigg|\leq\nonumber\\
&\hspace{0.8in}\sum_{N', N_1'}\Bigg| \sum_{\substack{n\sim N', n_1\sim N_1', n_2\sim N_2', 
n_3\sim N_3'\\ 
m-m_1 +m_2-m_3 =0, |m-n^2|\geq K}}c(n,\bar n)A_{n,m,\bar n,\bar m}\Bigg|\label{(4.9)}\\
&\hspace{0.65in}+\sum_{N', N_1'}\Bigg|\sum_{\substack{n\sim N', n_1\sim N_1', n_2\sim N_2', n_3\sim N_3'\\
 m-m_1 +m_2-m_3 =0, |m-n^2|<K, \\|m_1-n_1^2|\geq K}}c(n,\bar n)A_{n,m,\bar n,\bar m}\Bigg|\label{(4.10)}\\
&\hspace{0.65in}+\Bigg|\sum_{\substack{n\leq N, n_1\leq N_1, n_2\sim N_2', n_3\sim N_3'\\
m-m_1+m_2 -m_3=0, |m-n^2|<K, |m_1-n_1^2|<K}} c(n,\bar n)A_{n,m,\bar n,\bar m}\Bigg|.  \label{(4.11)}
\end{align}

Making a further decomposition according to which $|n^2-n_1^2+n^2_2 - n^2_3|>10K$ or
$|n^2-n^2_1+n^2_2-n_3^2|\lesssim 10K$ in \eqref{(4.11)}, the contribution of 
\eqref{(4.11)} may be evaluated by bounding
\begin{align}
&\sum_{N', N_1'}\Bigg|\sum_{\substack{
n\sim N', n_i\sim N_i', |n^2-n_1^2 + n^2_2 -n_3^2|>10K\\
m-m_1+m_2-m_3=0, |m-n^2|<K, |m_1-n_1^2|<K}}c(n,\bar n)A_{n,m,\bar n,\bar m}\Bigg|\label{(4.12)}\\
&\hspace{0.4in}+\Bigg|\sum_{\substack{n\leq N, n_1\leq N_1, n_2\sim N_2', n_3\sim N_3'\\ m-m_1+ m_2-m_3=0}}
c_K(n, \bar n)\, A_{n,m,\bar n,\bar m}\Bigg|\label{(4.13)}
\end{align}
where in \eqref{(4.13)} we replaced $\hat v$ by $\hat v \chi_{|m-n^2|\leq K}$  and  $\hat v_1$  by 
$\hat v_1 \chi_{|m_1-n_1^2|<K}$ (noting that the norm $\ltriple\,\cdot\,\rtriple$ 
is unconditional).

Note that if $n_2=n_3$ and $|n^2 -n_1^2+n_2^2 -n_3^2|\leq 10K$, then either $n=n_1$ or 
$N'+N_1'\lesssim (N_2')^{10^{-3}}$.  Hence, \eqref{(4.13)} is bounded by
\begin{align}
&\hspace{0.50in}\sum_{N', N_1'}\hspace{0.28in} \Bigg|\sum_{\substack{n\sim N', n_i\sim N_i', n_2\not= n_3\\ m-m_1+m_2 -m_3 =0}} c_K(n, \bar n)A_{n,m,\bar n,\bar m}\Bigg|\label{(4.14)}\\
&\hspace{0.10in}+\sum_{N', N_1' \lesssim (N_2')^{10^{-3}}} \Bigg|
\sum_{\substack{ n\sim N',  n_1\sim N_1', n_2\sim N_2'\\ m-m_1+ m_2-m_3 =0}} 
c_K (n, n_1, n_2, n_2)A_{n,m,(n_1,n_2,n_2),\bar m}\Bigg|\label{(4.15)}\\
&\hspace{0.10in}+\Bigg| \sum_{\substack{ n\leq N, n_2 \sim N_2'\\ m-m_1+m_2-m_3=0}} c(n, n, 
n_2, n_2) A_{n,m,(n,n_2,n_2),\bar m}\Bigg| \label{(4.16)}
\end{align}

Let $$\sigma_{n, N_2'} =\sum_{n_2\sim N_2'} \frac 1{n^2_2} c(n, n, n_2, n_2)= O(1)$$ and estimate \eqref{(4.16)} by
\begin{align}
&\sum_{N'}\Bigg|\sum_{n\sim N'}\int_0^1 \overline{\hat v(n)(\tau)} \hat v_1 (n)(\tau) 
\Bigg[ \sum_{n_2\sim N_2'} c(n, n, n_2, n_2)|\hat u(n) (\tau)|^2-\sigma_{n, N_2'}\Bigg] d\tau\Bigg|\label {(4.17)}\\
&\hspace{1.2in}+\Bigg|\sum_{n\leq N} \Big(\int^1_0 \overline{\hat v(n)(\tau)} \, 
\hat v_1 (n)(\tau) d\tau\Big)\sigma_{n, N_2'}\Bigg|.\label{(4.18)}
\end{align}

In view of the above observations, our estimate of $\ltriple \eqref{(4.1)}\rtriple$ reduces 
to establishing bounds on \eqref{(4.9)}, \eqref{(4.10)}, \eqref{(4.12)}, \eqref{(4.14)}, \eqref{(4.15)}, 
\eqref{(4.17)} and \eqref{(4.18)}; this will be the topic of the following two sections.

The choice of $T$ is dictated by \eqref{(4.18)}, and we treat this term first.  Indeed, taking $T$ sufficiently small, Lemma \ref{Lemma3} gives
\begin{align}
\nonumber &\Bigg|\sum_{n\leq N, m}\, \overline{\hat v(n, m)}\hat v_1(n, m)
 \sigma_{n, N_2'}\Bigg|\\
&\hspace{0.6in}\lesssim \sum_{n\leq N, m} |\hat v(n, m) | \, |\hat v_1(n, m)|
\lesssim T= o\Big(\frac 1{\log N_*}\Big).\label{(4.19)}
\end{align}
Evaluating the summation over dyadic $N_2'\leq N_*$ then allows us to conclude that the contribution of \eqref{(4.18)} can be estimated by  $o(1)$.

\bigskip

\section
{Multilinear estimates (I)}

In this section, we obtain bounds on the terms \eqref{(4.9)}, \eqref{(4.10)} and \eqref{(4.12)}.  The remaining 
terms will be treated in the next section.  

We begin with the contribution of \eqref{(4.9)}.  Fix the values $N', N_1'$ and rewrite the inner sum in 
\eqref{(4.9)} as
\be \label{(5.1)} 
\int_B\int_0^1 \, \overline{(P_{n\sim N'} v)} (P_{n_1\sim N_1'} v_1) 
\,  \overline {(P_{n_2\sim N_2'} u)} (P_{n_3\sim N_3'} u)dtdx,
\ee
where $\hat v(n, m)=0$ for $|m-n^2|<K$.

It follows from the definition of the $\ltriple\,\cdot\,\rtriple$ norm that 
$$
\Vert P_{n\sim N'} v\Vert_{0, \frac 13} < K^{-\frac 17} \ltriple P_{n\sim N'} v\rtriple.
$$
Moreover, by (viii) from Lemma \ref{Lemma4}, applied with $b=\frac 14+\frac{3\epsilon}{2}$, $\epsilon=10^{-6}$, we 
therefore obtain
\be\label{(5.2)}
\Vert P_{n\sim N'} v\Vert_{L_x^{\frac 3{1+\epsilon}}L_t^{\frac 2{1-\epsilon}}} < K^{-\frac 17} 
\ltriple P_{n\sim N'}v\rtriple.
\ee

Also by (viii) of Lemma \ref{Lemma4}
\be
\label{(5.3)}
\Vert P_{n_1\sim N_1'}v_1\Vert_{L_x^{\frac 3{1+\epsilon}}L_t^{\frac 2
{1-\epsilon}}} \lesssim \ltriple P_{n\sim N_1'} v_1\rtriple.
\ee

To estimate the contributions of $u_2$ and $u_3$ to \eqref{(5.1)}, we use the apriori bound given by 
Lemma \ref{Lemma5} with $q=\frac{2}{\epsilon}$, where $\epsilon=10^{-6}$ as before.  In particular, we 
may ensure that
\be\label{(5.4)}
\max\{\Vert P_{n_2\sim N_2'} u\Vert_{L^{6-\epsilon}_xL^q_t} ,\,\, \Vert P_{n_3\sim N_3'} u \Vert_{L^{6-\epsilon}_xL^q_t}\}<(N_2')^{10^{-6}}
\ee
outside an exceptional set of measure at most $\exp(-c(N_2')^{c10^{-6}})$
in the initial datum $\phi$.  Taking $p=\frac {6}{1-2\epsilon}$, we then obtain
\be\label{(5.5)}
\Vert P_{n_2\sim N_2'} u\Vert_{L_x^pL_t^q} \lesssim (N_2') ^{10^{-6}+\frac{3}{6-\epsilon}-\frac{1}{2}+\epsilon}.
\ee

Hence, from \eqref{(5.2)}-\eqref{(5.5)} and recalling that $K=(N_2')^{10^{-3}}$ and $\epsilon=10^{-6}$, it follows that
\begin{align}
\eqref{(5.1)}&<K^{-\frac 17}(N_2')^{10^{-6}+\frac{3}{6-\epsilon}-\frac{1}{2}+\epsilon}\ltriple P_{n\sim N'}v\rtriple \ 
\ltriple P_{n\sim N_1'} v_1\rtriple\nonumber\\
&<(N_2')^{-\frac 12 10^{-4}}\ltriple P_{n\sim N'} v\rtriple \ 
\ltriple P_{n\sim N_1'} v_1\rtriple.\label{(5.6)}
\end{align}

To complete the estimate of the contribution of \eqref{(4.9)}, it remains to perform dyadic summation 
over $N', N_1', N_2'$ and $N_3'$, with $N_2'\geq N_3'$.  Note that from the definition of the $\ltriple\,\cdot\,\rtriple$ 
norm, one has
\be\label{(5.7)}
\ltriple v\rtriple^2 \sim \sum_{N'} \ltriple P_{n\sim N'} v\rtriple^2.
\ee

In view of \eqref{(5.6)}, there is of course no problem with the summation over values of $N_2'$ and 
$N_3'$, and we may also assume $\max \{N', N_1'\}>\exp((N_2')^{10^{-5}})$.  Consider the case $N' \geq N_1'$.  
If $N'\sim N_1'$, the estimate follows by using Cauchy-Schwarz and \eqref{(5.7)} for $v$ and $v_1$.  Assume 
now that $N'> 4N_1'$ holds.  We estimate the contribution of such terms to \eqref{(4.9)} by
\be
\gamma\int \Big[\sum_{n\sim N'} |\hat v(n)(t)|\Big] \ \Big[\sum_{n_1\sim N_1'} |\hat v_1 (n_1)(t)|\Big]\
\Big[{\sum_{n_2\sim N_2'} |\hat u(n_2)(t)|\Big] \ \Big[ \sum_{n_3\sim N_3'}|\hat u(n_3)(t)}|\Big] dt \label{(5.8)}
\ee
with
$$
\gamma =\max_{n\sim N', n_i\sim N_i'} |c(n, n_1, n_2, n_3)|.
$$
We then have the bound
\begin{align}
\eqref{(5.8)} &\leq \gamma\cdot (N'N_1'N_2'N_3')^{\frac 12} \Vert v\Vert_{L^2_t L_x^2}\Vert v_1\Vert_{L^2_tL^2_x} 
\Vert u\Vert^2_{L_t^\infty L_x^2}\nonumber\\
&\leq \gamma\cdot(N'N_1'N_2'N_3')^{\frac 12} \ltriple v\rtriple \ 
\ltriple v_1\rtriple \ \Vert u\Vert^2_{L^\infty_tL_x^2}.\label{(5.9)}
\end{align}

To evaluate $\gamma$, we write
\be
\label{(5.10)}
\int_B e_n e_{n_1} e_{n_2} e_{n_3}dx=\int^1_0 \sin(n\pi r)\sin (n_1\pi r)\vp (r) dr
\ee
with $\vp(r) =\frac {\sin(\pi n_2 r)} r\cdot \frac {\sin(\pi n_3 r)} r$, and note that integration by parts gives
\begin{align*}
\int_0^1 \cos((n\pm n_1) \pi r) \, \vp(r) dr
&=-\frac 1{\pi(n\pm n_1)} \int_0^1 \vp' (r) \sin((n\pm n_1) \pi r)dr\\
&< O\Big(\frac {\Vert\vp'\Vert_{L^\infty}}{(n\pm n_1)^2}\Big) \\
&< O\Big (\frac {(N_2')^2 N_3'}{(N')^2}\Big),
\end{align*}
where the last line follows from $N'>4N_1'$. Hence
$$
\eqref{(5.9)} \lesssim \frac {(N_2')^4}{N'}.
$$

Summing \eqref{(5.9)} 
over dyadic $N',N'_1,N'_2$ and $N'_3$ satisfying $N'>\max\{\exp((N'_2)^{10^{-5}}),$ $4N'_1\}$ and $N'_3\leq N'_2$, the 
contribution of \eqref{(5.8)} is then bounded by
\begin{align*}
\sum_{N',N'_2} \frac{(N'_2)^{4}(\log N'_2)(\log N')}{N'}
&<\frac{1}{N'_2},
\end{align*}
which completes the estimate of the contribution of \eqref{(4.9)}.

Since $v$ and $v_1$ play the same role, the 
same argument also takes care of contribution of \eqref{(4.10)}.

We now address the contribution of \eqref{(4.12)}.  Since the estimate relies only on $X_{s, b}$ norms, 
Fourier restrictions are not an issue.  Note that since $|m-n^2|<K$, $|m_1-n_1^2|<K$ and 
$|n^2- n_1^2+n_2^2 -n_3^2|> 10 K$, at least one of the conditions
$$|m_2-n_2^2|>K\quad\textrm{or}\quad |m_3-n_3^2|>K$$ holds.  

Assume 
\be\label{(5.11)}
|m_2-n_2^2|\gtrsim |m_3-n_3^2|>K.
\ee
We distinguish several cases.

\noindent
\underline{\it Case 1}: $N'+N_1' <(N_2')^3$.

Consider the expression
\be\label{(5.12)}
\sum_{\substack {n\sim N', n_i\sim N_i', |n^2-n_1^2 +n_2^2 -n_3^2|\gtrsim K\\ 
m-m_1+m_2-m_3 =0\\ |m_2-n_2^2|\gtrsim |m_3-n_3^2|>K}}
c(n, \bar n)A_{n,m,\bar n,\bar m}
\ee
where we assume $\ltriple v\rtriple, \ltriple v_1\rtriple \leq 1$ and, according 
to Proposition \ref{Proposition6}, that $\Vert u\Vert_{\frac 12-, \frac 34-}<O(1)$.

The restriction $|n^2-n_1^2+n_2^2-n_3^2|\gtrsim K$  in \eqref{(5.12)} may be removed arguing as follows: Let $0\leq \psi\leq 1$ 
be a parameter, and replace $\hat v(n, m)$ by $$e(n^2\psi)\hat v(n, m),$$ and 
$\hat v_1(n_1, m_1)$, $\hat u(n_2, m_2)$ and $\hat u(n_3, m_3)$ by $$e(n_1^2 \psi)\hat v_1(n_1, m_1), \quad
 e(n_2^2\psi)\hat u(n_2, m_2)\quad\textrm{and}\quad e(n_3^2\psi)\hat u(n_3, m_3),$$ respectively.  The restriction 
$|n^2-n_1^2 +n_2^2-n_3^2|\lesssim K$ may then be achieved by taking a suitable average over $\psi$.

It thus suffices to bound the expression
\be\label{(5.13)}
\sum_{\substack {n\sim N', n_i\sim N_i',\\m-m_1+m_2-m_3=0\\|m_2-n_2^2|\gtrsim |m_3-n_3^2|>K}}
  c(n, \bar n) \, \overline{\hat v(n, m)} \, \hat v_1(n_1, m_1) \, 
\overline {\hat u_2 (n_2, m_2)} \, \hat u_3(n_3, m_3)
\ee
with $\ltriple v\rtriple, \ltriple v_1\rtriple\leq 1$, 
$\Vert u_2\Vert_{\frac 12-,\frac 34-}< O(1)$ and $\Vert u_3\Vert_{\frac 12-, \frac 34-}<O(1)$.

To bound this quantity, we re-express \eqref{(5.13)} as
\be\label{(5.14)}
\int_B\int_0^1 \, \overline {K^{-\epsilon}(P_{n\sim N'}v)} \, 
[K^{-\epsilon}(P_{n_1\sim N_1'}v_1)] \, \overline
{K^{2\epsilon}(P_{n_2\sim N_2'}u_2)} \, [P_{n_3\sim N_3'} u_3]dtdx
\ee
with $\epsilon =10^{-6}$, where to simplify notation we have suppressed an additional Fourier restriction on the $u_2$ and $u_3$ factors.

Since the norm $\ltriple\,\cdot\,\rtriple$ indeed controls the norm $\Vert \cdot \Vert_{0, \frac 12-}$, and 
the condition $K^\epsilon >(N')^{\frac 1310^{-9}}$ holds by assumption, we may apply inequality (iii) 
of Lemma \ref{Lemma4} to obtain
\be
\label{(5.15)}
\Vert K^{-\epsilon} P_{n\sim N'} v\Vert_{L_x^3L_t^{4-}}<O(1)
\ee
and, similarly,
\be\label{(5.16)}
\Vert K^{-\epsilon} P_{n_1\sim N_1'} v_1\Vert_{L_x^3L^{4-}_t}< O(1).
\ee

On the other hand, using the Fourier restriction due to \eqref{(5.11)},
\begin{align*}
\Vert P_{n_2\sim N_2'} u_2\Vert_{\frac 12-, \frac 58}&< K^{-\frac 1{16}}\Vert u_2\Vert_{\frac 12-, \frac 34-},\\
\Vert P_{n_2\sim N_2'} u_2\Vert_{\frac 12+\epsilon, \frac 58}&< K^{-\frac 1{16}}(N_2')^\epsilon \Vert u_2\Vert_{\frac 12-, \frac 34-}.
\end{align*}
Applying Lemma \ref{Lemma4}, (v), it follows that
\be\label{(5.17)}
\Vert P_{n_2\sim N_2'} u_2\Vert_{L^p_x L^q_t}\lesssim K^{-\frac 1{16}}(N_2')^\epsilon \Vert u_2\Vert_{\frac 12-, \frac 34-}
\ee
with
\be\label{(5.18)}
p=\frac 6{1-\frac\epsilon 2},\quad q=\frac 4{1-\frac\epsilon 2}.
\ee

In addition, Lemma \ref{Lemma4} (v) gives
\be\label{(5.19)}
\Vert P_{n_3\sim N_3'} u_3\Vert_{L^{6-}_x L^{4-}_t} <O(1).
\ee

Combining \eqref{(5.15)}-\eqref{(5.19)}, we obtain
$$
|\eqref{(5.14)}|\lesssim (N_2')^{-\frac {10^{-3}}{16}+10^{-6}+2\cdot 10^{-9}} \lesssim (N_2')^{-10^{-5}}.
$$

Summing in $N'$ and $N'_1$ now gives the bound
\begin{align*}
(N'_2)^{-10^{-5}}[\log(N_2')]^2\lesssim (N'_2)^{-\frac{10^{-5}}{2}}
\end{align*}
for the contribution of these terms to \eqref{(4.12)}.

\noindent
\underline{\it Case 2}: $N'+N_1' >(N_2')^3$ and $n\not=n_1$.

In this case, we have
$$
N'+N_1'-(N_2')^2<|n^2-n_1^2+n_2^2-n_3^2|< |n_2^2-m_2| +|n_3^2-m_3|+2K
$$
and hence
$$
|n_2^2-m_2|>\frac 13(N'+N_1').
$$
This clearly allows us to repeat the analysis of Case 1 with $\frac{1}{3}(N'+N_1')$ in place of $K$, giving again the bound
\begin{align*}
(N_2')^{-10^{-5}}.
\end{align*}

\noindent
\underline{\it Case 3}: $N'=N_1'>(N_2')^3, n=n_1$.

Proceeding as in Case 1 above, we obtain
\be\label{(5.20)}
\sum_{n\sim N', n_i\sim N_i', m-m_1+m_2-m_3=0} c(n, n, n_2, n_3) A_{n,m,(n,n_2,n_3),\bar m} 
\ee
with $|m_2-n_2^2|>K$. Rewrite \eqref{(5.20)} as
\be\label{(5.21)}
\int_B\int_0^1 \Big[\sum_{n\sim N'} \, \overline {\hat v(n)}\, \hat v_1 
(n) e_n^2\Big]\, \overline{(P_{n_2\sim N_2'}\, u_2)} (P_{n_3\sim N_3'} u_3)dtdx.
\ee

Now, observe that it follows from \eqref{(5.17)} and \eqref{(5.18)} that
\be\label{(5.22)}
\Vert P_{n_2\sim N_2'} u_2\Vert_{L^p_x L^q_t}<(N_2')^{-10^{-5}},
\ee
while \eqref{(5.19)} gives
\be\label{(5.23)}
\Vert P_{n_3\sim N_3'} u_3\Vert_{L^{6-}_x L^{4-}_t} <O(1).
\ee

On the other hand, since $e_n^2(x)\leq \frac 1{|x|^2}$, the first factor in the integrand of \eqref{(5.21)} is 
bounded by
\be\label{(5.24)}
\frac 1{|x|^2} \Big(\sum_{n\sim N'} |\hat v(n)|^2\Big)^{\frac 12} \Big(\sum_{n\sim N'} |\hat v_1(n)|^2\Big)^{\frac 12}
\ee
where for any $q_1<\infty$ one has
$$
\Big \Vert \Big(\sum_n|\hat v(n)|^2\Big)^{\frac 12}\Big
\Vert_{L_t^{q_1}} =\Vert v\Vert_{L_t^{q_1}L_x^2}
\lesssim \lVert v\rVert_{0, \frac 12 -} \lesssim \ltriple v\rtriple,
$$
with the analogous bound for $v_1$.

It then follows that 
\be\label{(5.25)}
\Vert\eqref{(5.24)} \Vert_{L_x^{\frac 32-} L_t^{q_1}}<O(1).
\ee

Combining \eqref{(5.25)} with \eqref{(5.22)} along with \eqref{(5.23)} and summing in $N'$ now gives that the contribution of \eqref{(5.20)} is bounded by 
\begin{align*}
(N_2')^{-10^{-5}},
\end{align*}
completing the bound in this case.

\section
{Multilinear estimates (II)}

In this section, we estimate the remaining contributions, those of 
\eqref{(4.14)}, \eqref{(4.15)} and \eqref{(4.17)}.
This will involve a different type of analysis than that used in the 
previous section; in particular we will make essential use of several
further probabilistic considerations related to the solution map.

We begin with \eqref{(4.14)}.  Rewrite this quantity as a sum over $N', N_1'$ of
\be\label{(6.1)}
\Big|\int_0^1\Big[\sum_{n\sim N', n_i\sim N_i', n_2\not= n_3} c_K(n, \bar n)\, \overline{\hat v(n)} \, \hat v_1(n_1)\,
\overline{\hat u(n_2)} \, \hat u(n_3) \Big]dt\Big|.
\ee
Note that in the sum we necessarily have $n\not= n_1$, since otherwise
$$
N_2'+N_3' \leq |n_2^2 -n_3^2|\leq 10K = 10(N_2')^{10^{-3}},
$$
giving a contradiction.  

Hence, it follows that
$$
N'+N_1' \leq |n^2-n_1^2| \leq K+8(N_2')^2 < 9(N_2')^2.
$$

We first examine the contribution for $n\not= n_3$.
Denote $N', N_i'$ by $N, N_i$ for simplicity.
Since $\Vert v\Vert_{L^2_{t,x}}\lesssim 1$, it follows from Cauchy-Schwarz that \eqref{(6.1)} is
bounded by the $L^2_t$-norm of
\begin{align*}
&\Big[\sum_n\Big|\sum_{n_1, n_2, n_3}\hat v_1(n_1) \, \overline{\hat u(n_2)} \hat u(n_3) c_K(n, n_1, n_2,
n_3)\Big|^2\Big]^{\frac 12}\\
&\hspace{1.2in}\leq \Big[\sum_{n_1, n_1'} |\hat v_1(n_1)| \, |\hat v_1 (n_1')|\Big|
\sum_{\substack{n, n_2, n_2',\\ n_3, n_3'}}\, B_{n,\bar n,\bar n'}\Big|\Big]^{\frac 12}
\end{align*}
where
\begin{align*}
B_{n,\bar n,\bar n'}=\overline{\hat u(n_2)} \, \hat u(n_3) \hat u(n_2') \, 
\overline{\hat u(n_3')} \, c_K(n, n_1, n_2, n_3) c_K(n, n_1', n_2', n_3')
\end{align*}
and again by Cauchy-Schwarz
\begin{align}
\nonumber&\Big[\sum_{n_1}|\hat v_1(n_1)|^2\Big]^{\frac 12}
\Big[\sum_{n_1\not= n_1'}| B_{n,\bar n,\bar n'}|^2\Big]^{\frac 14}
+
\Big[\sum_{n_1} |\hat v_1(n_1)|^2\Big]^{\frac 12} \Big[\max_{n_1=n_1'} |
B_{n,\bar n,\bar n'}|\Big]^{\frac 12}\\
&\hspace{1.1in}\leq \Vert v_1(t)\Vert_{L^2_x} \Big\{\Big[\sum_{n_1\not= n_1'} |B_{n,\bar n,\bar n'}|^2\Big]^{\frac 14} +
\max_{n_1=n_1'} |B_{n,\bar n,\bar n'}|^{\frac 12}\Big\}.\label{(6.2)}
\end{align}

Since $\Vert v_1\Vert_{L_t^qL_x^2}\lesssim \ltriple v_1\rtriple<O(1)$ 
for all $q<\infty$, it suffices to bound
\be\label{(6.3)}
\Vert\{\cdots\}\Vert_{L_t^4}
\ee
where $\{\cdots\}$ is the quantity appearing in \eqref{(6.2)}.  

Note that \eqref{(6.3)} involves only the truncated solution $u$ with 
initial data $\phi=\phi_\omega$, and we view $u$ as a random variable 
of $\omega$.  For fixed $t$, the distribution of $u_\omega(t)$ is given
by a Gaussian Fourier series $$\sum_n \frac {g_n(\omega)}{n}
e_n$$ with $\{g_n\}$ as a sequence of IID normalized complex Gaussians.
This fact is essential to our analysis in this section.

For sufficiently large $q$, we may estimate
$$
\bigg(\mathbb E_\omega \big[\lVert \{\cdots\}\rVert_{L_t^4}^q\big]\bigg)^{\frac 1q} \leq \Vert \ \Vert \{\cdots\}\Vert_{L_\omega^q} \Vert_{L_t^4}\leq \max_{0\leq t\leq
1}\Vert\{\cdots\}\Vert_{L_\omega^ q}
$$
and, fixing $t$, we accordingly write 
\begin{align}
\nonumber&\Vert\{\cdots\}\Vert_{L_\omega^q}\leq\\
&\hspace{0.3in}\Big\{\sum_{n_1\not= n_1'} \Big\Vert \sum_{n, n_2, n_3, n_2', n_3'} \, 
\frac{\overline{g_{n_2}}}{n_2} \, 
\frac {g_{n_3}}{n_3} \, 
\frac {g_{n_2'}}{n_2'} \, \frac 
{\overline{g_{n_3'}}}{n_3'} c_K(n, \bar n)
c_K(n, \bar n')\Big\Vert^2_{L_\omega^{q/2}}\Big\}^{\frac 14}\label{(6.4)}\\
&\hspace{0.3in}+\Big\Vert \max_{n_1}\Big|\sum_{n, n_2, n_3, n_2', n_3'} \frac{\overline{g_{n_2}}}{n_2} \, \frac {g_{n_3}}{n_3}\,
\frac {g_{n_2'}}{n_2'} \, \frac {\overline{g_{n_3'}}}{n_3'}\,  c_K(n, \bar n) c_K(n, n_1, n_2', n_3')
\Big|\Big\Vert^{\frac 12}_{L_\omega^{q/2}}
\label{(6.5)}
\end{align}

We first analyze \eqref {(6.4)} by considering several cases, recalling that
$n_2\not= n_3$ and $n_2' \not= n_3'$.

\noindent
\underline{\it Case 1}: $n_2\not= n_2', n_3\not=n_3'$.

In this case, we note that the bound
$$
c_K(n, n_1, n_2, n_3) \lesssim N_3 \chi_{[| n^2-n_1^2+n_2^2-n_3^2|<K]}
$$
gives the estimate
\begin{align*}
&\mathbb E_\omega \bigg[\bigg|\sum_{n,n_2,n_3,n'_2,n'_3} \frac{\overline{g_{n_2}}}{n_2}\frac{g_{n_3}}{n_3}\frac{g_{n'_2}}{n'_2}\frac{\overline{g_{n'_3}}}{n'_3}c_K(n,\bar n)c_K(n,\bar n')\bigg|^2\bigg]\\
&\hspace{0.4in}\lesssim \frac 1{N_2^4}\sum_{n_2, n_2', n_3, n_3'} 
\Big(\sum_n \chi_{|n^2-n_1^2+n_2^2-n_3^2|<K} \, \chi_{|n^2-(n_1')^2 +(n_2')^2
-(n_3')^2|<K}\Big)^2\\
&\hspace{0.4in}\lesssim \frac{\sqrt K}{N_2^{4}} \sum_{n, n_2, n_2', n_3, n_3'} \chi_{|n^2-n_1^2+n_2^2-n_3^2|<K} \,
\chi_{|n^2-(n_1')^2 +(n_2')^2 -(n_3')^2|<K}.
\end{align*}
For the summation over $n_1$ and $n'_1$ in \eqref{(6.4)}, this gives the bound
\begin{align}\label{(6.6)}
\nonumber &\frac{\sqrt K}{N_2^{4}} \bigg|\bigg\{ (n, n_1, n_1', n_2, n_2', n_3, n_3') \, : \, n_i, n_i' \sim N_i, n\not= n_1, n_1',\,\textrm{and}\\
&\hspace{0.83in} |n^2-n_1^2 +n_2^2 -n_3^2|<K,\, |n^2-(n_1')^2+(n_2')^2 -(n_3')^2|<K \bigg\}\bigg|
\end{align}

Fix values of $k$, $k'$ with $|k|, |k'|<K$, and evaluate the number of 
solutions of the equations
\be\label{(6.7)}
\begin{cases} n^2-n_1^2+n_2^2-n_3^2 =k\\ n^2-(n_1')^2 +(n_2')^2 -(n_3')^2 =k'
\end{cases}
\ee
in the variables $n, n_1, n_1', n_2, n_2', n_3$ and $n_3'$.

For this purpose, further fix $n_2, n_2', n_3$.  Since $n\pm n_1$ are divisors of $k-n_2^2 +n_3^2\not=0$, this specifies $n, n_1$ up to $N_2^{0+}$ possibilities.
Next, writing
\be\label{(6.8)}
(n_1')^2 +(n_3')^2 =n^2 +(n_2')^2 -k'
\ee
the usual bounds for the number of $\mathbb Z^2$-points on circles (and circle arcs) imply that \eqref{(6.8)} has at
most $N_3^{0+}$ solutions in $(n_1', n_3')$.

Summarizing, this proves that
\be\label{(6.9)}
\eqref{(6.6)} <\sqrt K N_2^{-4} K^2 N_2^{2+} N_3 < N_2^{-1/2}
\ee

\noindent
\underline{\it Case 2}: $n_2=n_2', n_3\not= n_3'$.

We obtain
\begin{align}\label{(6.10)}
\nonumber &\mathbb E_\omega\bigg[\bigg|\sum_{n,n_2,n_3,n'_3} \frac{|g_{n_2}|^2}{(n_2)^2}\frac{g_{n_3}}{n_3}\frac{\overline{g_{n'_3}}}{n'_3}c_K(n,\bar n)c_K(n,n_1',n_2,n_3')\bigg|^2\bigg] \\
&\hspace{0.4in}\lesssim \frac 1{N^4_2} \sum_{n_3, n_3'}\Big(\sum_{n, n_2}
\chi_{|n^2-n_1^2+n_2^2 -n_3^2|<K} \cdot \chi_{|n^2 -(n_1')^2 +n^2_2 -(n_3')^2|<K}\Big)^2
\end{align} 
and since the number of $(n,n_2)$-terms in the inner sum is at most $KN_2^{0+}$ (for given $n_1, n_1', n_3, n_3'$), we obtain
\be\label{(6.11)}
\eqref{(6.10)} \ll N_2^{-4+} K \sum_{n, n_2, n_3, n_3'}\chi_{|n^2-n_1^2+n_2^2 -n_3^2|<K} \, \chi_{|n^2 -(n_1')^2 +n^2_2 -(n_3')^2|<K}.
\ee

Taking the summation of \eqref{(6.11)} over $n_1$ and $n_1'$ then gives the bound $$N_2^{-4+} K^3 N_2N_3 <N_2^{-1}$$ for the contribution to the sum in \eqref{(6.4)}.

\noindent
\underline{\it Case 3}: $n_2\not= n_2', n_3=n_3'$.

In place of \eqref{(6.10)}, we get
\be\label{(6.12)}
\frac 1{N^4_2} \sum_{n_2, n_2'} \Big(\sum_{n, n_3} 
\chi_{|n^2-n_1^2+n_2^2 -n_3^2|<K} \cdot \chi_{|n^2 -(n_1')^2 +(n_2') -n_3^2| <K}\Big)^2.
\ee

Writing $n^2-n_1^2 +n_2^2 -n_3^2=k$, $|k|<K$, in the inner sum, it follows that
$n\pm n_3$ divides $k+n_1^2 -n_2^2 \not=0$, since $n\not= n_3$.
Thus, there are at most $KN_2^{0+}$ terms in the inner sum and we obtain the bound
$$N_2^{-4+} K^3 N_2^2 N_3 <N_2^{-\frac 12}$$
for the contribution of \eqref{(6.12)} to the sum in \eqref{(6.4)}.

\noindent
\underline{\it Case 4}: $n_2=n_2', n_3=n_3'$.

In this case, the inner sum in \eqref{(6.4)} becomes
\be\label{(6.13)}
N_2^{2} N_3^{-2} \sum_{n, n_2, n_3} c_K (n, n_1, n_2, n_3) c_K(n, n_1', n_2, n_3).
\ee

It follows from the definition of $c_K$ that the quantity $\eqref{(6.13)}$ vanishes unless
$$
|n_1^2 -(n_1')^2|<2K 
$$
holds; note that this implies $N_1=O(K)$, since $n_1\not= n_1'$. Thus
\begin{align*}
\eqref{(6.13)} & < N_2^{-2} \Big|\Big\{ (n, n_2, n_3): n_2\sim N_2, n_3\sim N_3\text { and }
|n^2+n_2^2-n_3^2|\lesssim K^2\Big\}\Big|\\
&\lesssim N_2^{-2}K^2N_3N_2^{0+}\\
&<N_2^{-\frac 34}
\end{align*}
and the corresponding contribution to \eqref{(6.4)} is bounded by $N_2^{-\frac 14}$.

The considerations in Cases 1-4 take care of the estimate of \eqref{(6.4)}.

We next consider the estimate of \eqref{(6.5)}.  Note that the analogues of Cases 1, 2 and 3 in this setting are captured by the previous analysis, since we did not use the condition $n_1\not= n_1'$.

To treat the estimate in the analogue of Case 4, we bound the contribution to \eqref{(6.5)} by
\begin{align*}
&(\log N_1)\big[\max_{n_1\sim N_1} N_2^{-2} N_3^{-2}\sum_{n, n_2, n_3} 
c_K(n, n_1, n_2, n_3)^2\big]^{\frac 12}\\
&\hspace{0.8in}\lesssim (\log N_1) N_2^{-1} \Big[\max_{n_1\sim N_1} 
\sum_{n, n_2, n_3} \chi_{|n^2-n_1^2 +n^2_2-n_3^2 |<K}\Big]^{\frac 12}\\
&\hspace{0.8in}\lesssim (\log N_1)N_2^{-1} (KN_3 N_2^{0+})^{\frac 12} \\
&\hspace{0.8in}\lesssim N_2^{-\frac 13}.
\end{align*}

This completes the treatement of Case 4 for the estimate of \eqref{(6.5)}.  
Combining the estimates of \eqref{(6.4)} and \eqref{(6.5)} then completes 
the analysis of the contribution of terms where $n\neq n_3$.  

We now consider the terms for which $n_3=n$.  Note that, since under this 
condition we have
$$
|n_1^2 -n_2^2|\lesssim K=(N_2)^{10^{-3}},
$$
it also follows that $n_1=n_2$ in this setting.
We then estimate the contribution to \eqref{(6.1)} by
$$
\min (N, N_1) \int\Big[\sum_{n\sim N, n_1\sim N_1} |\hat v(n)| \, |\hat v_1(n_1)|
\, |\hat u(n)|\, |\hat u(n_1)|\Big]dt
$$
which, after using Cauchy-Schwarz, is in turn estimated by
$$
\min(N, N_1)\int\Vert P_{n\sim N} v\Vert_{L_x^2} \Vert P_{n\sim N}u\Vert_{L_x^2}
\Vert P_{n_1\sim N_1} v_1\Vert_{L^2_x} \Vert P_{n_1\sim N_1} u\Vert_{L_x^2} dt.
$$

Using H\"older and summing over $N$ and $N_1$, we obtain the bound
\begin{align}
\nonumber&\sum_{N, N_1} \min \{N, N_1\}\Vert P_{n\sim N} v\Vert_{L^2_{t,x}} \Vert P_{n\sim N} u\Vert_{L^6_tL^2_x} \\
&\hspace{1.2in}\cdot\Vert
P_{n_1\sim N_1} v_1\Vert _{L^6_tL^2_x} \Vert P_{n_1\sim N_1} u\Vert_{L^6_tL^2_x}.\label{(6.14)}
\end{align}

Moreover, since $\Vert \ \Vert_{L^q_tL_x^2}\lesssim \ltriple\,\cdot\,\rtriple$ holds 
for all $q$, and, by Lemma \ref{Lemma3}, 
$$\big(\sum_N\Vert P_{n\sim N} v\Vert^2_{L^2_{x, t}}\Big)^{\frac 12} =\Vert v\Vert_{L^2_{x, t}}
\lesssim \sqrt T \ltriple v\rtriple,$$ 
it follows from Cauchy-Schwarz that
\begin{align}
\nonumber \eqref{(6.14)}&\lesssim \sqrt T\Big\{\sum_N \Vert P_{n\sim N} u\Vert^2_{L^6_tL_x^2}\Big(\sum_{N_1} 
\min\{N, N_1\}\\
&\hspace{1.2in}\cdot\ltriple P_{n_1\sim N_1} v_1\rtriple\ 
\Vert P_{n_1\sim N_1} u\Vert_{L^6_tL_x^2}\Big)^2\Big\}^{\frac 12}.
\label{(6.15)}
\end{align}

To control the norms of projections of $u$ appearing in \eqref{(6.15)} we require the following
probabilistic estimate.
\begin{lemma}
\label{Lemma7}
Let $1\leq q<+\infty$ be given.  Then there exists $c>0$ such that for every $\lambda\geq 1$, one has
\begin{align}\label{(6.16)}
\mu_F^{(N)}\bigg(\bigg\{\phi:\max_N \,\,N^{1/2}\lVert P_{n\sim N}
u_\phi\rVert_{L_t^q L_x^2}>\lambda\bigg\}\bigg)\leq \exp(-c\lambda^c).
\end{align}
where the maximum is taken over dyadic integers $N$.
\end{lemma}

Assuming that Lemma \ref{Lemma7} holds, we use this bound to estimate \eqref{(6.15)} by
$$
\sqrt{T}\bigg\{\sum_N\Big(\sum_{N_1} \frac {\min \{N, N_1\}}{\sqrt{NN_1}}\ltriple P_{n_1\sim N_1} v_1\rtriple\Big)^2\bigg\}^{1/2} \lesssim
\sqrt{T}\ltriple v_1\rtriple.
$$
This leads to the bound $O(\sqrt T)$ on \eqref{(6.14)}.

This completes the analysis of the contribution of (4.14) except for the proof of Lemma $\ref{Lemma7}$, which we address presently.

\begin{proof}[Proof of Lemma $\ref{Lemma7}$]
We begin by noting that it suffices to establish
\begin{align}\label{(6.17)}
\mu_G(A_{\lambda})\leq \exp(-c\lambda^c).
\end{align}
with $A_{\lambda}:=\{\phi:\max_N N^{1/2}\lVert P_{n\sim N} u_\phi\rVert_{L_t^q([0,T_*);L_x^2(B))}>\lambda\}$.  
Indeed, arguing as in the proof of Lemma $\ref{Lemma5}$, \eqref{(6.17)} implies then an inequality of the type
\eqref{(6.16)}.

It therefore remains to establish \eqref{(6.17)}.  Toward this end, fixing $q_1>q$ and applying the Tchebychev
inequality and Plancherel identity followed by the Minkowski inequality, one obtains
\begin{align}
\nonumber \mu_G(A_{\lambda})&\leq \lambda^{-q_1}\bigg\lVert \max_N \,\,\bigg(N^{1/2}\lVert P_{n\sim N}
u_\phi\rVert_{L_t^qL_x^2}\bigg)\bigg\rVert_{L^{q_1}(d\mu_G)}^{q_1}\\
\nonumber &\lesssim \lambda^{-q_1}\bigg\lVert \max_N\,\, N^{1/2} \bigg(\sum_{n\sim N}
|\widehat{u}_\phi(n)|^2\bigg)^{1/2}\bigg\rVert_{L^{q_1}(d\mu_G)}\bigg\rVert_{L_t^{q}}^{q_1}\\
\nonumber &\lesssim \lambda^{-q_1}\bigg\lVert \max_N\,\, N^{1/2} \bigg(\sum_{n\sim N}
|\widehat{\phi}(n)|^2\bigg)^{1/2}\bigg\rVert_{L^{q_1}(d\mu_G)}^{q_1}\\
&\lesssim \lambda^{-q_1}\bigg\lVert \max_N\,\, N^{1/2}\bigg(\sum_{n\sim N}
\frac{|g_n(\omega)|^2}{n^2}\bigg)^{1/2}\bigg\rVert_{L_\omega^{q_1}}^{q_1}\label{(6.18)}
\end{align}
where we used the invariance of the Gibbs measure to obtain the third inequality.  

We therefore have
\begin{align*}
\eqref{(6.18)}&\lesssim \lambda^{-q_1}\bigg\{1+\bigg(\sum_{N}\bigg\lVert \sum_{n\sim N}
\frac{N}{n^2}(|g_n(\omega)|^2-1)\bigg\rVert_{L_\omega^{q_1}}\bigg)^{1/2}\bigg\}^{q_1}\\
&\lesssim \lambda^{-q_1}\bigg\{1+\bigg(\sum_N \frac{q_1}{\sqrt{N}}\bigg)^{1/2}\bigg\}^{q_1}\\
&\lesssim \bigg(\frac{\sqrt{q_1}}{\lambda}\bigg)^{q_1},
\end{align*}
where we used the estimate
\begin{align*}
\bigg\lVert \sum_{n\sim N} \frac{N}{n^2}(|g_n(\omega)|^2-1)\bigg\rVert_{L_\omega^{q_1}}\lesssim Nq_1\bigg(\sum_{n\sim
N} \frac{1}{n^4}\bigg)^{1/2}\lesssim \frac{q_1}{\sqrt{N}}
\end{align*}
which follows from \eqref{(2.7)}.  Optimizing the choice of $q_1$ (by essentially taking
$q_1=\lambda^2/2$; see, for instance, the proof of Lemma \ref{Lemma5}), now yields the desired claim.

This completes the proof of Lemma $\ref{Lemma7}$.
\end{proof}

It remains to bound the contributions of \eqref{(4.15)} and \eqref{(4.17)}.  We begin with \eqref{(4.15)}, 
for which we argue by expressing the inner sum in this expression as
$$
\sum_{n\sim N, n_1\sim N_1, n_2\sim N_2} c_K(n, n_1, n_2, n_2) \int_0^1 [\overline{\hat v(n)}
\hat v_1 (n_1) |\hat u(n_2)|^2]dt.
$$
Using Lemma $\ref{Lemma7}$, this is in turn bounded by
\begin{align*}
&N\int_0^1 \Big(\sum_{n\sim N} |\hat v(n)|\Big) \Big(\sum_{n_1\sim N_1}|\hat v_1(n_1)|\Big) \Big(\sum_{n_2\sim N_2}
|\hat u(n_2)|^2\Big)dt\\
&\hspace{0.4in}\leq N^{3/2} N_1^{1/2} \int_0^1\Vert P_{n\sim N} v\Vert_{L_x^2}
\Vert P_{n_1\sim N_1} v_1\Vert_{L_x^2} \, \Vert P_{n_2\sim N_2} u\Vert^2_{L^2_x} dt\\
&\hspace{0.4in}\lesssim (N_2)^{2\cdot 10^{-3}} \Vert P_{n\sim N} v\Vert_{L^4_tL^2_x} \, \Vert P_{n_1\sim N_1} v_1\Vert_{L^4_tL_x^2} \,
\Vert P_{n_2\sim N_2} u\Vert^2_{L^4_tL^2_x}\\
&\hspace{0.4in}\lesssim (N_2)^{2\cdot 10^{-3}-1}.
\end{align*}

We next consider \eqref{(4.17)}.  We use the Cauchy-Schwarz inequality to bound this expression by
\begin{align}
\nonumber
&\sum_{N}\Vert P_{n\sim N}v\Vert_{L^2_{t,x}}\Vert P_{n\sim N}v_1\Vert_{L_t^4L_x^2} \Big\Vert \max_{n\sim N} 
\Big|\sum_{n_2\sim N_2}
c(n, n, n_2, n_2)|\hat u(n_2)|^2-\sigma_{n, N_2}\Big|\, \Big\Vert_{L^4_t}\\
&\hspace{0.7in}\lesssim T^{1/2}\sup_{N}\Big\Vert \max_{n\sim N} \Big|\sum_{n_2\sim N_2} c(n, n, n_2, n_2)|\hat u(n_2)|^2-\sigma_{n, N_2}
\Big|\, \Big\Vert_{L^4_t}\label{(6.19)}
\end{align}
Recall that $$\sigma_n =\sigma_{n, N_2} =\mathbb E_\phi\Big[\sum_{n_2\sim N_2} c(n, n, n_2, n_2)
|\hat\phi(n_2)|^2\Big].$$

The bound on the second factor in \eqref{(6.19)} again follows from probabilistic considerations.  We have the following:

\begin{lemma}\label{Lemma8}
For $\lambda\gg 1$, we have for some constant $c>0$
\be
\mu_F\Big[\phi; \Big\Vert\max _n \Big|\sum_{n_2 \sim N_2}|\widehat{u_\phi(t)} (n_2)|^2 (c(n, n, n_2, n_2)-\sigma_n
\Big|\Big\Vert_{L_t^4} >\lambda\Big] \lesssim e^{-c\lambda^cN_2^c}\label{(6.20)}
\ee
\end{lemma}

\begin{proof}
It suffices again to prove \eqref{(6.20)} with $\mu_F$ replaced by the Gibbs measure $\mu_G$.
Proceeding as in Lemma 6.1, take $q_1=q_1(\lambda)$ and write
\begin{align*}
&\bigg\Vert\,\bigg\Vert\max_n\Big|\sum_{n_2 \sim N_2}|\widehat{u_\phi(t)} (n_2)|^2 (c(n, n, n_2, n_2)-\sigma_n\Big|\bigg\Vert_{L_t^4}\bigg\Vert_{L^{q_1}(\mu_G(d\phi))}\\
&\hspace{0.6in}\leq\bigg\Vert\,\bigg\Vert\max_n \Big|\sum_{n_2 \sim N_2}|\widehat{u_\phi(t)} (n_2)|^2 (c(n, n, n_2, n_2)-\sigma_n\Big| \bigg\Vert_{L^{q_1}(\mu_G(d\phi))}\bigg\Vert_{L_t^4}.
\end{align*}
Using the Gibbs measure invariance under the flow, the above is bounded by
\begin{align}
\nonumber&\Big\Vert\max_n\Big| \sum_{n_2\sim N_2} |\hat\phi(n_2)|^2 c(n, n, n_2,
n_2)-\sigma_n\Big|\Big\Vert_{L^{q_1}(\mu_G(d\phi))}\\
&\hspace{0.6in}\leq \Big\Vert\max_n\Big|\sum_{n_2\sim N_2} \ \frac{c(n, n, n_2, n_2)}{n_2^2} (|g_{n_2} (\omega)|^2-1)\Big|
\, \Big\Vert_{L^{q_1}(d\omega)}.\label {(6.21)}
\end{align}
Note that
\begin{align}
\nonumber c(n, n, n_2, n_2)&=\int_0^1 \sin^2(\pi nr) \, \frac{\sin^2(\pi n_2r)}{r^2} dr\\
&= \frac 12\int_0^1 \frac {\sin^2(\pi n_2r)}{r^2} dr -\frac 12\int^1_0 \cos(2\pi nr) \, \frac {\sin^2(\pi n_2r)}{r^2} dr.
\label{(6.22)}
\end{align}
The second term in \eqref{(6.22)} is bounded by $O\big(\frac{N_2^4}{N^2}\big)$ for $n>N$, and therefore its 
contribution to \eqref{(6.21)} is at most
\begin{align*}
O\Big(\frac{N_2^2}{N^2}\Big)\Big\Vert\sum_{n_2\sim N_2} (|g_{n_2}(\omega)|^2+1) \Big\Vert_{L^{q_1}(d\omega)} 
<O\Big(\frac{q_1N_2^3}{N^2}\Big)
<O(q_1  N_2^{-1})
\end{align*}
for $N> N_2^2$.

Hence, we may restrict $n$ in \eqref{(6.21)} to the range $n\leq N_2^2$ and get the bound
\begin{align*}
O(\log N_2) \max_{n<N^2_2} \Big\Vert\sum_{n_2\sim N_2} \frac {c(n, n, n_2,
n_2)}{n_2^2}(|g_{n_2}(\omega)|^2-1)\Big\Vert_{L^{q_1}(d\omega)}
<O(\log N_2)q_1 N_2^{-\frac 12}.
\end{align*}
Taking $q_1\sim \lambda N_2^{\frac 13}$ and applying Tchebycheff's inequality, \eqref{(6.20)} follows.
\end{proof}

Having estimated the contributions of \eqref{(4.14)}, \eqref{(4.15)} and \eqref{(4.17)}, this completes our analysis of the nonlinear term \eqref{(4.5)}.

\section
{Further probabilistic considerations}

Returning to the nonlinear term \eqref{(4.1)}, an inspection of the estimates in Section 5 and Section 6 -- including Lemma \ref{Lemma7} and Lemma \ref{Lemma8}
-- as well as the non-probabilistic inequality \eqref{(4.19)} which determines the size of $T$, gives the following statement.

\begin{proposition}
\label{Proposition9}
Let $T$ be as in \eqref{(4.2)} and take $M_i\leq N_i$ for $i=2, 3$, $M=M_2+M_3$.
Moreover, let $u=u_\phi$ denote the solution of some truncated equation \eqref{(1.1)}.
Then
\begin{align}
\nonumber &\Bigltriple \int_0^t e^{i(t-\tau)\Delta} P_N[(P_{N_1} U^1) \, 
\overline{(P_{M_2\leq n\leq N_2} u)}\, (P_{M_3\leq n\leq N_3}u)](\tau) dt\Bigrtriple\\
&\hspace{3.2in}\leq 10^{-3}\ltriple U^1\rtriple
\label{(7.1)}
\end{align}
holds for all $U^1$ for which the right side is finite, assuming that $\phi$ is restricted to the complement of 
an exceptional set of measure at most $\exp(-M^c)$ (with $c>0$ some
constant).
\end{proposition}

Note that for $M$ small, we have the bound (cf. \eqref{(5.1)})
\begin{align}
\nonumber &\sup_{\ltriple v\rtriple\leq 1} \Big(\int_B\int_0^1 |P_Nv| \, |P_{N_1} U^1| \, |P_{M_2} u| \, |P_{M_3} u| dxdt\Big)\\
\nonumber &\hspace{0.8in}\leq \sup_{\ltriple v\rtriple \leq 1}\big(\Vert v\Vert_{L_{t,x}^2}\Vert U^1\Vert_{L^2_{t,x}} 
\Vert P_M u\Vert^2 _{L^\infty_t L_x^\infty}\big)\\
\nonumber &\hspace{0.8in}\lesssim
T\ltriple U^1\rtriple M^3 \Vert u\Vert^2_{L^\infty_tL_x^2} \\
&\hspace{0.8in}\leq TM^3 \Vert\phi\Vert^2_{L_{x}^2} \ltriple U^1\rtriple,
\label{(7.2)}
\end{align}
where the second inequality follows from Lemma \ref{Lemma3} and the third inequality is a consequence of the conservation of the $L_x^2$ norm under the flow.

Recalling also the discussion in Section 4 on how to treat \eqref{(4.1)} with solutions $u^{(2)}$ and $u^{(3)}$ 
obtained from different truncations, we obtain 
\begin{proposition}\label{Proposition10}
Let $T$ be given by \eqref{(4.2)}. Then,
\begin{align}
\nonumber &\Bigltriple \int_0^t e^{i(t-\tau)\Delta} P_N[(P_{N_1}U^1) \, \overline{(P_{N_2} u^{(2)})}
\, (P_{N_3} u^{(3)})](\tau) dr\Bigrtriple \\
&\hspace{3.2in}\leq 10^{-3} \ltriple U^1\rtriple
\end{align}
holds for all $U^1$ for which the right side is finite.
Here $u^{(i)}|_{t=0}=P_{N^{(i)}} \phi$ satisfies the $N^{(i)}$-truncated equation 
$(i= 2, 3)$ and  we assume $\phi$ is outside an
exceptional set of measure at most $O\big(\exp(-T^{-c})\big)$ (independent of $U^1$).
\end{proposition}

As we will see in the next section, Proposition \ref{Proposition10} suffices to establish almost sure 
convergence of the sequence $\{u^N\}$ of truncated solutions of \eqref{(1.1)}, letting $N$ run over the 
integers $2^j$ (or any sufficiently rapidly increasing sequence).  However, the measure estimates do not 
quite suffice to conclude immediately the a.s. convergence of the full sequence, and an additional 
consideration is needed.  The idea is basically the following: in view of Proposition \ref{Proposition9}, 
we obtain the desired measure estimates for factors $P_{n\geq M_2} u^{(2)}$ and $P_{n\geq M_3}u^{(3)}$ 
provided that for instance $M$ satisfies $$M=M_2+M_3>\big(\log(N^{(2)}+N^{(3)})\big)^C$$ with $C$ an appropriate constant.  

It then remains to consider
\be\label{(7.4)}
\int_0^t e^{i(t-T)\Delta} P_N[(P_{N_1} U^1)\, \overline{(P_M u^{(2)})} (P_M u^{(3)})] (\tau)d\tau.
\ee
Fix some truncation $M<N^{(0)} < N^{(2)}, N^{(3)}$ and let $u^{(0)} =P_{N^{(0)}} u^{(0)}$
be the corresponding solution of \eqref{(1.1)} with initial data $u^{(0)}|_{t=0}=P_{N^{(0)}} \phi$.

We compare \eqref{(7.4)} with
\be\label{(7.5)}
\int_0^t e^{i(t-\tau)\Delta}P_N[(P_{N_1} U^1) \, \overline{(P_Mu^{(0)})}\, (P_M u^{(0)})]
(\tau) d\tau.
\ee
The difference between \eqref{(7.4)} and \eqref{(7.5)} may then be bounded by
\begin{align}
&\nonumber \Vert P_{N_1} U^1\Vert_{L_t^4L_x^2}\Big[\Vert P_M u^{(0)}-P_M u^{(2)}
\Vert_{L_t^4L_x^\infty} + \Vert P_Mu^{(0)} -P_Mu^{(3)}\Vert_{L_t^4L_x^\infty}\Big]\Vert P_M u^{(0)}\Vert_{L_t^4L_x^\infty}\\
&\nonumber \hspace{0.2in}+ \Vert P_{N_1} U^1\Vert_{L_t^4L_x^2}\Vert P_Mu^{(0)}- P_M u^{(2)}\Vert_{L_t^4L_x^\infty} \Vert P_M u^{(0)} - P_M u^{(3)}\Vert_{L_t^4L_x^\infty}\\
&\nonumber \hspace{0.6in}\lesssim\ltriple P_{N_1} U^1\rtriple M^3 \Vert P_M\phi\Vert_{L_x^2} \\
&\nonumber \hspace{1.2in}\cdot\Big[\ltriple P_M u^{(0)}- P_Mu^{(2)} \rtriple +\ltriple P_M u^{(0)} -P_M u^{(3)}\rtriple\Big]\\
&\hspace{0.8in}+\ltriple P_{N_1} U^1\rtriple M^3 \ltriple P_M u^{(0)} -P_M u^{(2)}\rtriple \  \ltriple P_M u^{(0)} -P_M u^{(3)}\rtriple.
\label{(7.6)}
\end{align}
The interest of this construction is that in order to bound \eqref{(7.5)}, only exceptional sets related to $u_\phi^{(0)}$ have to be removed, while
the prefactor $M^3$ in \eqref{(7.6)} is harmless in view of the smallness of 
$\ltriple P_M u^{(0)} -P_M u^{(i)}\rtriple$, $i=2,3$.
This will be made more precise in the next section.

\section
{Proof of the theorem}

In this section, we complete the proof of our main theorem.  Toward this end, let $1\ll N_0<N$ be given. Our goal is to compare the solutions $u^{N_0}$ and $u^N$ of
\begin{align}
\label{(8.1)}
&\begin{cases}
iu_t^{N_0}+\Delta u^{N_0} -P_{N_0} (u^{N_0} |u^{N_0}|^2)=0\\ u^{N_0}(0)=P_{N_0}\phi\end{cases}
\intertext{and}
\label{(8.2)}
&\begin{cases}
iu_t^N+\Delta u^N - P_N(u^N|u^N|^2)=0\\
u^N(0)=P_N\phi\end{cases}
\end{align}
on a time interval $I=[0, \eta]$ with $\eta>0$ a sufficiently small constant.

Let $1\ll M\leq N_0$ and set
\be\label{(8.3)} T=\frac c{\log M}
\ee
with $c>0$ taken as in Proposition \ref{Proposition10} with $N_i\leq M$ for $i=2,3$.

The argument consists of dividing $[0, \eta]$ into time intervals of size $T$ and applying Duhamel's 
formula on each of these subintervals in order to obtain recursive inequalities.

Taking $0\leq t\leq T$, we have
$$
u^N(t) =e^{it\Delta} (P_N\phi) +i\int_0^t e^{i(t-\tau)\Delta} P_N(u_N|u^N|^2)(\tau)d\tau
$$
and
\be\label{(8.4)}
P_M(u^N-u^{N_0})(t)= i\int_0^t e^{i(t-\tau)\Delta} [P_M(u^N|u^N|^2) -P_M(u^{N_0}|u^{N_0}|^2)|(\tau)d\tau.
\ee
We will make an estimate of the $\ltriple\,\cdot\,\rtriple$ norm of this quantity.

We first replace $u^N$ and $u^{N_0}$ in \eqref{(8.4)} by $P_M u^N$ and $P_M u^{N_0}$, respectively.
The $\ltriple \cdot \rtriple$ norm of the difference may then be estimated by
\begin{align}
\nonumber&\Big[\Vert u^{N_0} -P_M u^{N_0}\Vert_{L_x^{3+} L_t^6} +\Vert u^N-P_M u^N\Vert_{L^{3+}_xL^6_t}\Big]
\\
&\hspace{1.4in}\cdot \Big[\Vert u^{N_0}\Vert^2_{L_x^{6-}L^6_t}+\Vert u^N\Vert^2_{L^{6-}_x L^6_t}\Big]< M^{-\frac 14},\label{(8.5)}
\end{align}
where we have used the a priori bound given by Lemma \ref{Lemma5}; again, \eqref{(8.5)} holds outside an exceptional set of measure at most $O(e^{-M^c})$.

We then obtain
\begin{align}
&\ltriple P_M(u^N-u^{N_0})\rtriple<M^{-\frac 14}\nonumber\\
&\hspace{0.8in}+\Bigltriple \int_0^t e^{i(t-\tau)\Delta} [P_M(u^N-u^{N_0})|P_M u^N|^2](\tau)d\tau
\Bigrtriple\label{(8.6)}\\
&\hspace{0.4in}+\Bigltriple \int_0^t e^{i(t-\tau)\Delta} [(P_M u^{N_0}) \, \overline{\big(P_M (u^N -u^{N_0})\big)} \, (P_M u^N)] (\tau) d\tau\Bigrtriple\label{(8.7)}\\
&\hspace{0.8in}+\Bigltriple \int_0^t e^{i(t-\tau)\Delta}[|P_M u^{N_0}|^2 (P_M(u^N-u^{N_0})](\tau) d\tau
\Bigrtriple.\label{(8.8)}
\end{align}

In view of Proposition \ref{Proposition10}, each of the terms \eqref{(8.6)}, \eqref{(8.7)}, \eqref{(8.8)}
may be bounded by $10^{-3}\ltriple P_M(u^N-u^{N_0})\rtriple$, provided that $\phi$ is taken outside an exceptional set of measure at most
$\exp(-T^{-c})$.
Note that this set depends on $N_0$ and $N$.
The preceding discussion then implies that
\be\label{(8.9)}
\ltriple P_M(u^N- u^{N_0})\rtriple_{0,\frac{1}{2};T} < 2M^{-\frac 14}
\ee
and an application of Lemma \ref{Lemma1} gives the existence of some $t_1\in [\frac T2, T]$ such that
\be\label{(8.10)}
\Vert P_M(u^N-u^{N_0}) (t_1)\Vert_{L_x^2} < 2 C_1 M^{-\frac 14}.
\ee

Consider now the next time interval $[t_1, t_1+T]$ and write for each $t\in [0, T]$
\be\label{(8.11)}
u^N(t_1+t)= e^{it\Delta}\big(u^N(t_1)\big) +i\int_0^t e^{i(t-\tau)\Delta} P_N(u^N|u^N|^2)(t_1+\tau) d\tau.
\ee
Repeating the above argument, we obtain
\begin{align}
\nonumber \ltriple P_M(u^N-u^{N_0})(t_1 +\cdot)\rtriple &\leq C_0\Vert P_M(u^N-u^{N_0})(t_1)\Vert_{L_x^2}+M^{-\frac 14}\\
\nonumber &\hspace{0.6in}+ \tfrac{3}{10^{3}}\ltriple P_M(u^N-u^{N_0})(t_1+\cdot)\rtriple
\end{align}
and thus
\begin{align}
&\ltriple P_M (u^N-u^{N_0})(t_1+\cdot)\rtriple < 2(2C_0C_1+1)M^{-\frac 14}\label{(8.12)}
\end{align}
for $\phi$ outside a set of measure at most $\exp(-T^{-c})$.

Note that the value of $t_1$ in \eqref{(8.10)} depends on $\phi$ but this does not create problems with the estimates of the nonlinear terms.

Again by Lemma \ref{Lemma1}, \eqref{(8.12)} gives $t_2 \in [t_1+\frac T2, t_1+T]$ with 
\be \label{(8.13)}
 \Vert P_M(u^N-u^{N_0})(t_2)\Vert_{L_x^2} < 2C_1(C_0C_1+1)M^{-\frac 14}.
\ee

Repeating this argument recursively, we obtain times $t_{j+1} \in [t_j+\frac T2, t_j+T]$ for each $j\geq 1$,with
\be\label{(8.14)}
\Vert P_M(u^N-u^{N_0})(t_{j+1})\Vert_{L_x^2} \leq 2C_1 \Big[C_0\Vert P_M(u^N-u_{N_0})(t_j)\Vert_2+M^{-\frac 14}\Big].
\ee
Iterating the resulting bounds gives
\be\label{(8.15)}
\Vert P_M(u^N-u^{N_0})(t_j)\Vert_{L_x^2} < (4C_1C_0)^j M^{-\frac 14}<M^{-\frac 18},
\ee
since $j\leq T^{-1} \eta= c^{-1} \eta\log M$ by \eqref{(8.3)}, and provided that $\eta$ is chosen sufficiently small.

Since
$$
\ltriple P_M (u^N-u^{N_0})(t_j+\cdot)\rtriple < M^{-\frac 18}
$$
for each $j$, it follows from Lemma \ref{Lemma1} that $$
\frac 1T\int_I\Vert P_M(u^N- u^{N_0})\Vert^2_{L^2_x} dt \lesssim M^{-\frac 14}
$$
for each subinterval $I\subset [0, \eta]$ of size $T$.  We therefore obtain
\be\label{(8.16)}
\Vert P_M(u^N-u^{N_0})\Vert_{L^2_{t<\eta}L^2_x} \lesssim M^{-\frac 18}
\ee
for $\phi$ outside an exceptional set of measure at most 
$\frac 1T \exp(-T^{-c})<\exp(-T^{-c'})$ (depending on $N_0$ and $N$).

In view of the apriori bounds of Proposition $\ref{Proposition6}$ on the quantities
$\Vert u^{N_0}\Vert_{X^{s, b}}$ and $\Vert u^N\Vert_{X^{s, b}}$ for $s<\frac 12$ and 
$b<\frac 34)$ and interpolation arguments, the bound \eqref{(8.16)} also implies
\be\label{(8.18)}
\Vert u^N -u^{N_0}\Vert_{X^{s, b}[0, \eta]}< M^{-c(s, b)}
\ee
for $s<\frac 12$ and $b<\frac 34$.

To consider the interval $[\eta, 2\eta]$, we repeat the previous reasoning with $M$ replaced by $M_1=M^c$ and 
$T$ by $T_1 =\frac c{\log M_1 }$.
This gives
\be\label{(8.18b)}
\Vert u^N- u^{N_0}\Vert_{X^{s, b}[\eta, 2\eta]} < M_1^{-c(s, b)}
\ee
and so on.

Starting from $M=N_0$, the above argument shows that for any given time interval $[0, T]=I$ with $T<\infty$, 
the estimate
\be\label{(8.19)}
\Vert u^N-u^{N_0}\Vert_{X^{s, b}(I)}< N_0^{-c(s, b, T)}
\ee 
holds for $s< \frac 12, b<\frac 34$ and all $\phi$ outside a set of measure at most $e^{-(\log N_0)^c}$,
depending on $N$ and $N_0$.  This statement clearly implies convergence of the sequence $\{u^N\}$, $N=2^j$ in 
$$\bigcap_{s<\frac 12, b<\frac 34} X^{s, b}(I)$$ almost surely in $\phi$.

Since the series $$\sum_{N\in\mathbb Z_+} e^{-(\log N)^c}$$ diverges, this does not immediately imply 
the convergence of the full sequence.
In order to achieve this improvement of the convergence properties, we use the procedure discussed at 
the end of Section 7.

Toward this end, fix $N_0\gg 1$ and let $N$ range between $N_0$ and $2N_0$.
In \eqref{(7.5)}, let $u^{(0)} = u^{N_0}$,
and take $M$ as the truncation $$K=(\log N_0)^C$$ with $C$ a sufficiently large constant.

On the other hand, in the inequality \eqref{(8.6)}--\eqref{(8.8)} above, $\log M\sim \log N_0$.  
Recalling \eqref{(7.6)}, the estimation of \eqref{(8.6)}--\eqref{(8.8)} gives some additional terms:
\begin{align}
\nonumber &\ltriple P_M(u^N-u^{N_0})\rtriple \\
&\hspace{0.6in}<M^{-\frac 14}+10^{-3} \ltriple P_M(u^N-u^{N_0})\rtriple\nonumber\\
&\hspace{1.1in}+K^3 \Vert P_K\phi\Vert_{L_x^2} \ \ltriple P_K(u^N-u^{N_0})\rtriple \,\, \ltriple P_M(u^N-u^{N_0})\rtriple\nonumber\\
&\hspace{1.1in}+K^3\ltriple P_K(u^N-u^{N_0})\rtriple^2 \ \ltriple P_M(u^N-u^{N_0})\rtriple\nonumber\\
&\hspace{0.6in}< M^{-\frac 14}+\big[ 10^{-3}+K^3\Vert\phi\Vert_{L_x^2} \ltriple P_M(u^N-u^{N_0})\rtriple\nonumber\\
&\hspace{1.1in}+K^3 \ltriple P_M(u^N- u^{N_0})\rtriple^2 \big]\cdot
\ltriple P_M(u^N-u^{N_0})\rtriple.\label{(8.20)}
\end{align}

The inequality \eqref{(8.20)} holds for $\phi$ outside an exceptional set which is the union of a set of measure at most $e^{-(\log N_0)^c}$ depending on
$N_0$ and an exceptional set of measure at most $e^{-K^c}<N_0^{-2}$ depending on $N$.

Taking $\Vert\phi\Vert_{L_x^2}<K$ in \eqref{(8.20)} and recalling that $\log M\sim \log N_0$, we may again 
conclude \eqref{(8.9)}, which is now valid for all $N_0\leq N\leq 2N_0$ and $\phi$ outside an exceptional 
set of measure at most $e^{-(\log N_0)^{c}}$.

This completes the proof of the main theorem.


\begin{thebibliography}{99}

\bibitem{B-GAFA1} J. Bourgain. Fourier transform restriction phenomena for certain lattice subsets and applications to nonlinear 
evolution equations. I. Schr\"odinger equations. Geom. Funct. Anal. 3 (1993), no. 2, 107-156.

\bibitem{B-GAFA2} J. Bourgain. Fourier transform restriction phenomena for certain lattice subsets and applications to nonlinear evolution
equations. II. The KdV-equation. Geom. Funct. Anal. 3 (1993), no. 3, 209-262.

\bibitem{B1}  J. Bourgain. Periodic nonlinear Schr\"odinger equation in invariant measures. Comm. Math. Phys. 166 (1994), 1-24.

\bibitem{B2} J. Bourgain. On the Cauchy and invariant measure problem for the periodic Zakharov system. Duke Math. J. 76 (1994), no. 1, 175–202.

\bibitem{B3} J. Bourgain.  Invariant measures for the 2D-defocusing nonlinear Schr\"odinger equation. Comm. Math. Phys. 176 (1996), 421-445.

\bibitem{B4} J. Bourgain.  Invariant measures for NLS in infinite volume.  Comm. Math. Phys. 210 (2000), no. 3, 605–620.

\bibitem{B5}  J. Bourgain. Nonlinear Sch\"odinger equations. Hyperbolic equations and frequency interactions (Park City, UT, 1995), 3–157, IAS/Park
City Math. Ser., 5, Amer. Math. Soc., Providence, RI, 1999.

\bibitem{BB} J. Bourgain and A. Bulut. Gibbs measure evolution in radial nonlinear wave and Schr\"odinger equations on the ball.  Comptes Rendus
Math. 350 (2012) 11-12, pp. 571--575.

\bibitem{BB1} J. Bourgain and A. Bulut. Invariant Gibbs measure evolution for the radial nonlinear wave equation on the three dimensional ball.
Preprint. (2012)

\bibitem{BB2} J. Bourgain and A. Bulut. Almost sure global well posedness for radial NLS on the unit ball I: the 2D case.  Preprint. (2012)

\bibitem{BT12} N. Burq and N. Tzvetkov.  Random data Cauchy theory for supercritical wave equations. I. Local theory and II. A global existence
result.  Invent. Math. 173 (2008), no. 3, 449-475 and 477-496.
\end{thebibliography}
\end{document}